\documentclass[oneside,a4paper]{amsart}

\usepackage{graphicx}
\usepackage{amssymb}
\usepackage{mathrsfs}
\usepackage{color}
\usepackage[all]{xy}

\newcommand{\C}{\mathbb C}

\newcommand{\one}{\mathbf 1}

\newcommand{\CC}{\mathcal C}
\newcommand{\DD}{\mathcal D}

\newcommand{\EE}{\mathcal E}

\newcommand{\id}{{\mathsf{id}}}
\newcommand{\Id}{{\text{\rm Id}}}
\newcommand{\SSS}{\mathsf{S}}

\newcommand{\op}{{\text{op}}}
\newcommand{\pre}[1]{{}^{#1}\!}
\newcommand{\prevee}{{}^\vee\!}

\newcommand{\CAT}{{\mathscr C\!\mathrm{at}}}

\newcommand{\Set}{\mathrm{Set}}

\newcommand{\tensor}{\otimes}
\newcommand{\nattrans}{\Rightarrow}

\newcommand{\dinat}{\stackrel\bullet\nattrans}

\newcommand{\pt}{\text{pt}}

\newcommand{\pnt}{\star}
\renewcommand{\AA}{\mathcal A}
\newcommand{\BB}{\mathcal B}

\newcommand{\figdir}{pstex}
\newcommand{\pstex}[1]{\raisebox{-.45\height}{\input{\figdir/#1.pstex_t}}}

 \DeclareMathOperator{\Hom}{Hom}

 \DeclareMathOperator{\coev}{\mathsf{coev}}
 \DeclareMathOperator{\ev}{\mathsf{ev}}

 \DeclareMathOperator{\Nat}{Nat}

\renewcommand{\paragraph}[1]{\bigskip\noindent \emph{#1.}}

\newtheorem{thm}{Theorem}
\newtheorem{prop}[thm]{Proposition}

\newtheorem{lemma}[thm]{Lemma}

\theoremstyle{definition}
\newtheorem*{defn}{Definition}

\renewcommand{\phi}{\varphi}
\renewcommand{\epsilon}{\varepsilon}

\newcommand{\mytag}[1]{\tag{\text{\textsf{#1}}}}
\newcommand{\handtag}[1]{{\rm(\textsf{#1})}}
\newcommand{\bref}[1]{{\rm(\ref{#1})}}

\title{A diagrammatic approach to Hopf monads}
\author{Simon Willerton}
\email{s.willerton@shef.ac.uk}
\allowdisplaybreaks

\begin{document}

\begin{abstract}
Given a Hopf algebra in a symmetric monoidal category with duals, the
category of modules inherits the structure of a monoidal category with
duals.  If the notion of algebra is replaced with that of monad on a
monoidal category with duals then Brugui\`eres and Virelizier showed
when the category of modules inherits this structure of being monoidal
with duals, and this gave
rise to what they called a Hopf monad.  In this paper it is shown that
there are good diagrammatic descriptions of dinatural transformations
which allows the three-dimensional, object-free nature of their
constructions to become apparent.
\end{abstract}

\maketitle

\section*{Introduction}
\subsection*{Overview}
\thispagestyle{empty}
An algebra, i.e., a monoid in the category of vector spaces,
has an associated category of modules (or representations, if you
prefer).  A Hopf algebra is an algebra equipped with extra structure
which ensures that its category of modules inherits the monoidal
structure and
duals from the category of vector spaces.  These notions work similarly in braided monoidal categories
other than that of vector spaces.  A monad on a monoidal category can
be thought of as a generalization of an algebra (or monoid) in that category and has an associated
category of modules (also known, confusingly, as its category of
algebras).  
Brugi\`eres and Virelizier~\cite{BruguieresVirelizier:Hopf} defined, following Moerdijk~\cite{Moerdijk:MonadsTensorCategories},
a \emph{Hopf monad} structure which
ensures that a monad's category of modules is monoidal with duals.

The goal of this paper is to put some of the work of Brugui\`eres and
Virelizier into a diagrammatic context, which also means to put it
into appropriate framework of monoidal two-categories.  One of the
purposes of this was to make their constructions essentially
object-free.  To do this it was necessary to do various things
including using an object-free formulation of categories with duals,
which here means describing evaluation and coevaluation as dinatural
transformations and then extending string diagrammatics to include
dinatural transformations, something which appears to work rather
well.  Such dinatural transformations exist in the landscape of the
monoidal two-category of categories, and so are three-dimensional in
nature, thus are better manipulated, I would argue, using the
three-dimensional algebra presented here.

In terms of results on Hopf monads, many of the results are just
slight simplifications of those of Brugui\`eres and Virelizier.  The
example of a strong monoidal functor with a left adjoint is given a
more explicit treatment, this example being of
primary importance to me.  My
motivation lies in the specific case of a Hopf monad on the derived
category of coherent sheaves on a complex manifold; the monad coming
from a strong monoidal functor with a left adjoint.

\subsection*{Three-dimensional string diagrams}
The utility of string diagram notation in describing adjunctions and
monads is well-known (see for example
\cite{Lauda:FrobeniusAmbidextrous} though it undoubtably has its roots in the
Australian school).  Adjunctions and monads live in the two-category
of categories and their counterparts in the monoidal category (or
one-object two-category) of vector spaces are duals and algebras, and
quantum topologists know that these are well notated using string
diagrams, with the distinction between notation and application
becoming blurred in the case of knot invariants arising from ribbon
categories.  Low dimensional topology and low dimensional category
theory seem closely linked.
When one is considering monads on monoidal categories, one is led to
considering three-dimensional notation and this works similarly well.

There is some precedent in the use of surface
diagrams by the Australian school but this is not well
represented in the literature: see, for
example,~\cite{Street:FunctorialCalculus}.  There is also the mythical,
unavailable~\cite{McIntyreTrimble}, but I have not seen a copy.

\subsection*{Hopf monads}
A monad \(T\colon \CC\to\CC\) on a category is an endofunctor together
with a \emph{multiplication} natural transformation \(T^2\nattrans T\)
and a \emph{unit} natural transformation \(\Id_\C\nattrans T\)
satisfying some appropriate associativity and unital conditions.  One
then has the category of \(T\)-modules, consisting of pairs \((m,r)\)
where \(m\) is an object in \(\CC\) and \(r\colon T(m)\to m\) is an action map in a
suitable sense.  If \(\CC\) is a monoidal category then it is natural
to ask if the tensor product of any two \(T\)-modules can be given a
natural \(T\)-module structure.  For instance, if \(T\) is of the form
\(A\otimes {-}\) for some algebra object \(A\) of \(\CC\) then this
occurs when \(A\) has the structure of a bialgebra.  The question was
considered by Moerdijk \cite{Moerdijk:MonadsTensorCategories} and he
showed that lifting the monoidal structure on \(\CC\) to a monoidal structure
on the module category corresponds precisely to giving \(T\) some
extra structure to form what he called, following operad terminology,
a Hopf monad.  However in this arena this is not the best nomenclature
and a better term would be either \emph{bimonad}, by analogy with
bialgebras, or, more accurately, but less succinctly,
\emph{opmonoidal monad}, as the extra structure is the
same as making \(T\) into a monad in the category of opmonoidal
functors, as was observed by McCrudden
\cite{McCrudden:OpmonoidalMonads}.

A natural progression from this is to ask if duals from \(\CC\) lift
to duals on the category of modules, so if \(\CC\) is a monoidal
category with duals and \(M\) is object of \(\CC\) with an action of
the bimonad \(T\) on it, does the dual \(M^\vee\) also naturally have an
action of \(T\) on it?  In the case where \(T=A\otimes{-}\) is a monad
on, say, the category of vector spaces for a bialgebra \(A\), then
\(A\) being a \emph{Hopf} algebra, i.e., having an antipode, suffices
to make the module category have a lift of the duality on vector
spaces.  Brugui\`eres and Virelizier defined the notion of antipode
for a bimonad on a category with duals, such that it corresponds to
the duals lifting to the module category.

\subsection*{An example}
As explained in Section~4, a strong monoidal functor with a left
adjoint gives rise to a Hopf monad by composing it with its left
adjoint.  Suppose $G$ is a finite group, then a good example of such a
functor is the functor $\Delta^*$ from the category of representations
of $G\times G$ to the category of representations of $G$ defined by
restricting to the diagonal copy of $G$.  This is a strong monoidal
functor and has a left adjoint $\Delta_*$ which is induction from $G$
to $G\times G$ along the diagonal embedding.  Then $\Delta^*\Delta_*$
is a Hopf monad on the category of representations of $G$, in fact it
is basically the tensor-with-the-group-algebra functor, the group
algebra (with the conjugation action) is a Hopf algebra in this
category.  Similar examples arise in other areas such as when a
complex manifold takes the place of the finite group and the derived
category of coherent sheaves takes the place of the representation
category.  These examples naturally arise in topological quantum field
theory.  Details will appear elsewhere.

\subsection*{Synopsis} In the first section we recall the notion of
string diagrams and how they are useful for denoting adjoints.  We
then enhance the notation in a third dimension to denote products of
categories.  We then see how this is used with monoidal and opmonoidal
functors and natural transformations.  Finally opposite categories are
introduced to the notation.

In Section~2 we recall the string diagrammatic approach to monads and
see how Moerdijk's notion of bimonad fits in.  In Section~3 we recall
the notion of dinatural transformations and how they are used in the
definition of a monoidal category with duals.

Section~4 is where we get to the definition of Hopf monad and
Brugui\`eres and Virilizier's theorem about their categories of
modules being monoidal with duals.  Finally, following~\cite{BruguieresVirelizier:Hopf}, we see how a strong
monoidal functor with a left adjoint gives rise to a Hopf monad.

\subsection*{Terminology}
There is a split in terminology between category theorists and, say,
quantum algebraists, in that an object in a monoidal category with a
unital, associative multiplication is called a monoid by the former
and an algebra by the latter, primarily because the typical categories
for the two groups of mathematicians are respectively the category of sets and the
category of vector spaces.  We will stick to the latter terminology as
we are close to areas in which Lie algebra objects and universal
enveloping algebra objects are considered.  Similarly it makes sense
from this perspective to talk of \emph{modules} over a monad, rather
than \emph{algebras} over a monad as the category theorists prefer.
My apologies go to any reader to whom this seems ridiculous, but there
will be a good proportion of the audience to whom it will seem very sensible.

\subsection*{Acknowledgements}
Thanks to Alain Brugui\`eres, Eugenia Cheng, Aaron Lauda and Alexis
Virelizier for useful comments.

\section{The diagrammatics of natural transformations}
In this section we introduce the basic string diagram notation for
natural transformations.  In particular we will be interested in
monoidal categories and opmonoidal functors, so will need to extend
the notation in an extra dimension.  We will also need to consider
contravariant functors so we end this section with a way of denoting
opposite categories.

\subsection{String diagrams}
Firstly, here is a very quick reminder on the use of string diagrams to
represent natural transformations.  These diagrammatics have a history
going back to Feynmann and Penrose, but were formalized in the context
of monoidal categories by Joyal and Street \cite{JoyalStreet}, and this formalism extends
to arbitrary two-categories; though here we will
just be interested in the two-category of categories, functors and
natural transformations.  The
idea of string diagrams for natural transformations
is that the usual globular pictures of functors and natural transformations are
replaced by their Poincar\'e duals, so that categories are represented by
two-cells, functors by one-cells, and natural transformations by
zero-cells.  The basic example is as follows:
 \[\raisebox{-.45\height}{\input{\figdir/H1_1_0.pstex_t}}\qquad\text{becomes}\qquad\raisebox{-.45\height}{\input{\figdir/H1_1_1.pstex_t}}.\]
A more complicated example is the following:
 \[\raisebox{-.45\height}{\input{\figdir/H1_1_0b.pstex_t}}\qquad\text{becomes}\qquad\raisebox{-.45\height}{\input{\figdir/H1_1_1b.pstex_t}}.\]
Note that the identity functor is omitted from the string diagram
notation.

The standard example of the utility of string diagrams is with regard
to adjoint functors and this will turn out to be useful to us later.
So suppose that \(F\colon \CC\to\DD\) and \(U\colon \DD \to \CC\) form
an adjunction \(F\dashv U\), then we have the unit and counit natural
transformations, \(\epsilon \colon F\circ U\nattrans \Id_\DD\) and
\(\eta \colon \Id_\CC\nattrans U\circ F\), which are drawn as follows:
 \[\epsilon\equiv\raisebox{-.45\height}{\input{\figdir/H1_1_2.pstex_t}};\qquad\eta\equiv{\raisebox{-.45\height}{\input{\figdir/H1_1_3.pstex_t}}}.\]
Then the required conditions on the unit and counit are drawn as
 \[\raisebox{-.45\height}{\input{\figdir/H1_1_4.pstex_t}}=\raisebox{-.45\height}{\input{\figdir/H1_1_5.pstex_t}}\quad\text{and}\quad
 \raisebox{-.45\height}{\input{\figdir/H1_1_6.pstex_t}}=\raisebox{-.45\height}{\input{\figdir/H1_1_7.pstex_t}},\]
where the vertical line marked with \(F\) means the identity natural
transformation on the functor \(F\).

\subsection{Monoidal categories and monoidal functors}
The notation above can be enhanced to also encode cartesian products
of categories, so can denote, for instance, functors of the form
\(\otimes\colon \CC\times\CC\to\CC\).   In this context we can think
of the two-category  \(\CAT\) with its cartesian product as being a
one-object three-category, and so
should expect to have to use three-dimensions in our notation.

We will use the direction out of the page for the `product direction'.
So, for example, given functors \(G,H\colon \CC\times \CC'\to \CC''\)
and \(K\colon \CC''\to \CC''\), we denote \(G\) and \(K\circ H\) as
follows.
\[\raisebox{-.45\height}{\input{\figdir/H1_2_2a.pstex_t}}\quad\raisebox{-.45\height}{\input{\figdir/H1_2_2b.pstex_t}}\]
A natural transformation \(\theta\colon G \nattrans K\circ H\) will
then be denoted
\[\raisebox{-.45\height}{\input{\figdir/H1_2_2c.pstex_t}}\quad\text{or}\quad\raisebox{-.45\height}{\begin{picture}(0,0)%
\includegraphics{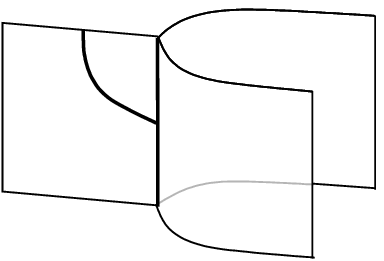}%
\end{picture}%
\setlength{\unitlength}{4144sp}%
\begingroup\makeatletter\ifx\SetFigFont\undefined%
\gdef\SetFigFont#1#2{%
  \fontsize{#1}{#2pt}%
  \selectfont}%
\fi\endgroup%
\begin{picture}(1730,1162)(2242,-1014)
\end{picture}%
},\]
the latter being used if the labels are clear from the
context.

Note that diagrams are read front to back, right to left and bottom to top.

% The product \(\CC\times \CC'\) of two categories will be drawn as two
% parallel sheets,
%  \[\pstex{H1_2_1},\]
% and a functor of the form \( will
% be drawn as follows:
%  \[\pstex{H1_2_2}.\]
% Note that diagrams are read right to left, bottom to top and front to back.

\subsubsection{Monoidal categories}
A monoidal category \((\CC,\otimes,\one)\) consists, as is well known,
of a
category \(\CC\), a functor \(\otimes\colon \CC\times \CC\to
\CC\), and a unit object \(\one\) --- here considered as a functor
\(\one\colon \pnt\to \CC\) from the one object, one morphism category
\(\pnt\) ---  together with an associativity natural transformation
\(\alpha\colon \otimes \circ (\Id\times \otimes)\nattrans \otimes
\circ (\otimes \times \Id)\),
and unit natural isomorphisms \(\nu_l\colon {\tensor}\circ
(\one\times \Id_\CC)\nattrans \Id_\CC\) and \(\nu_r\colon {\tensor}\circ
(\Id_\CC\times \one)\nattrans \Id_\CC\). The natural transformations have to satisfy the pentagon and
triangle relations.

The associativity will be drawn as
  \[\alpha\equiv\raisebox{-.45\height}{\begin{picture}(0,0)%
\includegraphics{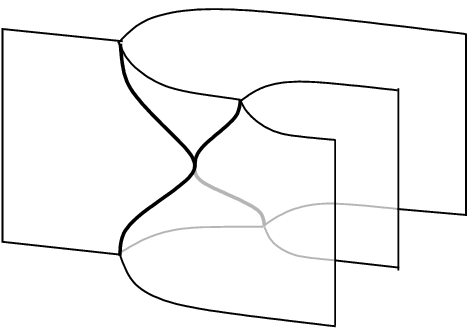}%
\end{picture}%
\setlength{\unitlength}{4144sp}%
\begingroup\makeatletter\ifx\SetFigFont\undefined%
\gdef\SetFigFont#1#2{%
  \fontsize{#1}{#2pt}%
  \selectfont}%
\fi\endgroup%
\begin{picture}(2143,1474)(2423,-1302)
\end{picture}%
},\]
but coherence means that notationally we can draw three-fold tensor
products with the understanding that to make sense they must be
resolved into the composition of two-fold tensor products, but that, up
to canonical identification, it is independent of the choice of
resolution, so we could instead draw the following triple tensor
product \({-}\otimes{-}\otimes{-}\colon \CC\times \CC\times \CC\to \CC\).
 \[\raisebox{-.45\height}{\begin{picture}(0,0)%
\includegraphics{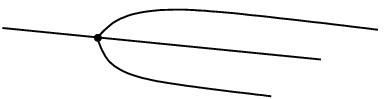}%
\end{picture}%
\setlength{\unitlength}{4144sp}%
\begingroup\makeatletter\ifx\SetFigFont\undefined%
\gdef\SetFigFont#1#2{%
  \fontsize{#1}{#2pt}%
  \selectfont}%
\fi\endgroup%
\begin{picture}(1740,420)(2836,-224)
\end{picture}%
}\]

The
category \(\pnt\) is the unit object for the product \(\times\) on
\(\CAT\); we will make the canonical identifications  \(\CC\times
\pnt\cong \CC\cong \pnt\times \CC\) and thus allow ourselves to denote
\(\pnt\)  by the empty surface.  So for instance the unit
\(\one\colon \pnt\to \CC\) will be denoted by the picture on the left,
from which the picture on the right will be understood.
  \[\raisebox{-.45\height}{\begin{picture}(0,0)%
\includegraphics{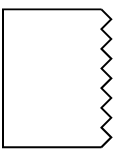}%
\end{picture}%
\setlength{\unitlength}{4144sp}%
\begingroup\makeatletter\ifx\SetFigFont\undefined%
\gdef\SetFigFont#1#2{%
  \fontsize{#1}{#2pt}%
  \selectfont}%
\fi\endgroup%
\begin{picture}(930,658)(1247,-344)
\put(1310,122){\makebox(0,0)[lb]{\smash{{\SetFigFont{8}{9.6}{\color[rgb]{0,0,0}$\CC$}%
}}}}
\end{picture}%
}\equiv\raisebox{-.45\height}{\input{\figdir/H1_2_5a.pstex_t}}.\]

The unit natural isomorphisms \(\nu_l\colon {\tensor}\circ
(\one\times \Id_\CC)\nattrans \Id_\CC\) and \(\nu_r\colon {{}\tensor{}}\circ
(\Id_\CC\times \one)\nattrans\Id_\CC\) are drawn as
 \[\nu_l\equiv\raisebox{-.45\height}{\begin{picture}(0,0)%
\includegraphics{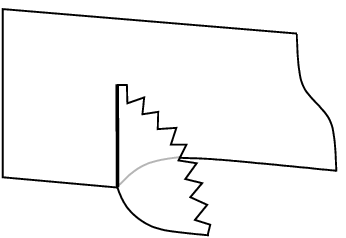}%
\end{picture}%
\setlength{\unitlength}{4144sp}%
\begingroup\makeatletter\ifx\SetFigFont\undefined%
\gdef\SetFigFont#1#2{%
  \fontsize{#1}{#2pt}%
  \selectfont}%
\fi\endgroup%
\begin{picture}(1549,1058)(1311,-2110)
\end{picture}%
}\qquad\text{and}\qquad\nu_r\equiv\raisebox{-.45\height}{\begin{picture}(0,0)%
\includegraphics{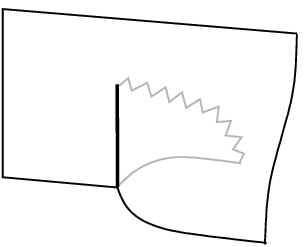}%
\end{picture}%
\setlength{\unitlength}{4144sp}%
\begingroup\makeatletter\ifx\SetFigFont\undefined%
\gdef\SetFigFont#1#2{%
  \fontsize{#1}{#2pt}%
  \selectfont}%
\fi\endgroup%
\begin{picture}(1369,1100)(1311,-2152)
\end{picture}%
}.\]
The inverses of these are drawn the same but the other way up.

\subsubsection{Monoidal and opmonoidal functors}
\label{Subsection:MonoidalFunctors}
If \((\CC,\otimes,\one)\) and \((\DD,\otimes,\one)\) are monoidal
categories then a monoidal functor \(M\colon \DD\to \CC\) ---
sometimes called a weak- or lax-monoidal functor --- means a functor
equipped with a natural family of morphisms \(M(d_1)\otimes M(d_2)\to
M(d_1\otimes d_2)\), parameterized by pairs of objects, and a morphism
\(\one\to M(\one)\) satisfying coherence conditions.  In other words, there are natural transformations
\(\sigma_M^2\colon{\tensor}\circ(M\times M)\nattrans M\circ \otimes\) and
\(\sigma_M^0\colon\one\nattrans M\circ\one\), satisfying some
constraints.  These natural transformations will be drawn as follows:
  \[\sigma_M^2\equiv\raisebox{-.45\height}{\input{\figdir/H1_2_8.pstex_t}}\qquad\text{and}\qquad\sigma_M^0\equiv\raisebox{-.45\height}{\begin{picture}(0,0)%
\includegraphics{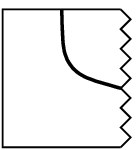}%
\end{picture}%
\setlength{\unitlength}{4144sp}%
\begingroup\makeatletter\ifx\SetFigFont\undefined%
\gdef\SetFigFont#1#2{%
  \fontsize{#1}{#2pt}%
  \selectfont}%
\fi\endgroup%
\begin{picture}(609,664)(979,-343)
\put(1081,119){\makebox(0,0)[lb]{\smash{{\SetFigFont{8}{9.6}{\color[rgb]{0,0,0}$M$}%
}}}}
\end{picture}%
}.\]
The conditions that they are required to satisfy are the
diagrammatically appealing:
  \[\raisebox{-.45\height}{\begin{picture}(0,0)%
\includegraphics{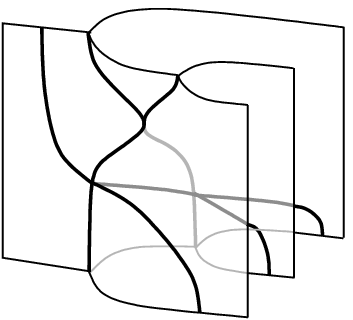}%
\end{picture}%
\setlength{\unitlength}{4144sp}%
\begingroup\makeatletter\ifx\SetFigFont\undefined%
\gdef\SetFigFont#1#2{%
  \fontsize{#1}{#2pt}%
  \selectfont}%
\fi\endgroup%
\begin{picture}(1583,1434)(2667,-1245)
\end{picture}%
}=\raisebox{-.45\height}{\begin{picture}(0,0)%
\includegraphics{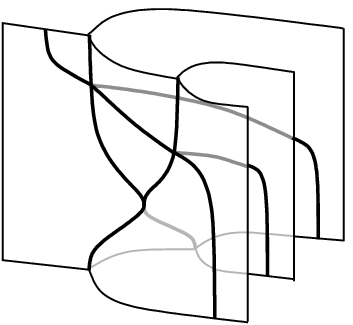}%
\end{picture}%
\setlength{\unitlength}{4144sp}%
\begingroup\makeatletter\ifx\SetFigFont\undefined%
\gdef\SetFigFont#1#2{%
  \fontsize{#1}{#2pt}%
  \selectfont}%
\fi\endgroup%
\begin{picture}(1583,1456)(2667,-899)
\end{picture}%
}\mytag{Mondl1}\label{Monoidal1}\]
  \[\raisebox{-.45\height}{\begin{picture}(0,0)%
\includegraphics{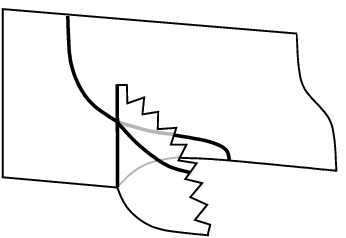}%
\end{picture}%
\setlength{\unitlength}{4144sp}%
\begingroup\makeatletter\ifx\SetFigFont\undefined%
\gdef\SetFigFont#1#2{%
  \fontsize{#1}{#2pt}%
  \selectfont}%
\fi\endgroup%
\begin{picture}(1549,1058)(1311,-2110)
\put(1710,-1302){\makebox(0,0)[lb]{\smash{{\SetFigFont{8}{9.6}{\color[rgb]{0,0,0}$M$}%
}}}}
\end{picture}%
}=\raisebox{-.45\height}{\begin{picture}(0,0)%
\includegraphics{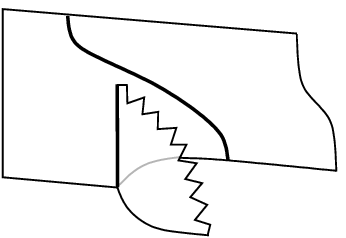}%
\end{picture}%
\setlength{\unitlength}{4144sp}%
\begingroup\makeatletter\ifx\SetFigFont\undefined%
\gdef\SetFigFont#1#2{%
  \fontsize{#1}{#2pt}%
  \selectfont}%
\fi\endgroup%
\begin{picture}(1549,1058)(1311,-2110)
\put(2033,-1344){\makebox(0,0)[lb]{\smash{{\SetFigFont{8}{9.6}{\color[rgb]{0,0,0}$M$}%
}}}}
\end{picture}%
}\mytag{Mondl2}\label{Monoidal2}
  %\quad\text{and}\quad
  \]\[
   \raisebox{-.45\height}{\begin{picture}(0,0)%
\includegraphics{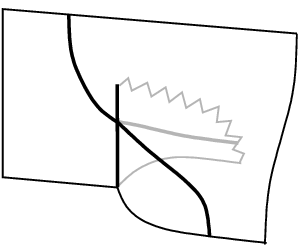}%
\end{picture}%
\setlength{\unitlength}{4144sp}%
\begingroup\makeatletter\ifx\SetFigFont\undefined%
\gdef\SetFigFont#1#2{%
  \fontsize{#1}{#2pt}%
  \selectfont}%
\fi\endgroup%
\begin{picture}(1369,1100)(1311,-2152)
\put(1683,-1278){\makebox(0,0)[lb]{\smash{{\SetFigFont{8}{9.6}{\color[rgb]{0,0,0}$M$}%
}}}}
\end{picture}%
}=\raisebox{-.45\height}{\begin{picture}(0,0)%
\includegraphics{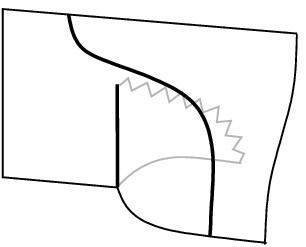}%
\end{picture}%
\setlength{\unitlength}{4144sp}%
\begingroup\makeatletter\ifx\SetFigFont\undefined%
\gdef\SetFigFont#1#2{%
  \fontsize{#1}{#2pt}%
  \selectfont}%
\fi\endgroup%
\begin{picture}(1369,1100)(1311,-2152)
\put(1924,-1310){\makebox(0,0)[lb]{\smash{{\SetFigFont{8}{9.6}{\color[rgb]{0,0,0}$M$}%
}}}}
\end{picture}%
}.\mytag{Mondl3}\label{Monoidal3}\]
The first condition together with associativity of the tensor product
means that again we can unambiguously draw triple
tensor products, with the understanding that such a triple tensor
product should be resolved into one of the two compositions of
ordinary tensor products.  Thus we think of the two natural
transformations pictured as a single natural transformation
\(({-}\otimes{-}\otimes{-})\circ(M\times M \times M)\nattrans
M\circ ({-}\otimes{-}\otimes{-})\), drawn as
\[\raisebox{-.45\height}{\begin{picture}(0,0)%
\includegraphics{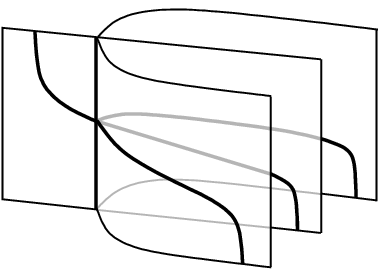}%
\end{picture}%
\setlength{\unitlength}{4144sp}%
\begingroup\makeatletter\ifx\SetFigFont\undefined%
\gdef\SetFigFont#1#2{%
  \fontsize{#1}{#2pt}%
  \selectfont}%
\fi\endgroup%
\begin{picture}(1740,1203)(2836,-1007)
\end{picture}%
}.\]

A functor \(M\) as above is said to be \emph{strong} monoidal if both the
natural transformations drawn above are natural \emph{isomorphisms}.

Similarly an opmonoidal functor \(Q\colon \CC\to \DD\) --- sometimes
called a comonoidal functor --- is a functor equipped with natural
transformations \(\sigma^Q_2\colon Q\circ \otimes\nattrans \tensor\circ(Q\times Q)\)
and \( \sigma^Q_0\colon Q\circ\one\nattrans \one\), drawn as
  \[\sigma^Q_2\equiv\raisebox{-.45\height}{\input{\figdir/H1_2_16.pstex_t}}\qquad\text{and}\qquad\sigma^Q_0\equiv\raisebox{-.45\height}{\begin{picture}(0,0)%
\includegraphics{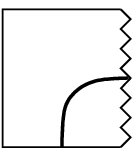}%
\end{picture}%
\setlength{\unitlength}{4144sp}%
\begingroup\makeatletter\ifx\SetFigFont\undefined%
\gdef\SetFigFont#1#2{%
  \fontsize{#1}{#2pt}%
  \selectfont}%
\fi\endgroup%
\begin{picture}(610,664)(979,-353)
\put(1191,-71){\makebox(0,0)[lb]{\smash{{\SetFigFont{8}{9.6}{\color[rgb]{0,0,0}$Q$}%
}}}}
\end{picture}%
},\]
which satisfy the above relations  inverted, which will be called (\textsf{Opmondl1--3}).

Note the important situation in which \(U\) and \(F\) form an adjoint pair
\(F\dashv U\) of
functors between monoidal categories.  Then there is a
bijection between pairs of natural transformations making \(U\)
monoidal and pairs of natural transformations making \(F\) opmonoidal;
or more informally, \(U\) is monoidal if and only if \(F\) is
opmonoidal.  To see this, suppose that \(U\) is monoidal, then an
opmonoidal structure is defined on \(F\) using the following two
natural transformations:
  \[\raisebox{-.45\height}{\input{\figdir/H1_2_18.pstex_t}}\qquad\text{and}\qquad\raisebox{-.45\height}{\input{\figdir/H1_2_19.pstex_t}}.\]
The proof that they satisfy the requisite
relations is easily derived diagrammatically.
\subsubsection{Opmonoidal natural transformations}\label{Section:OpMonNT}
We will need the notion of an opmonoidal natural transformation.  So
suppose that $P,Q\colon \CC\to \DD$ are two opmonoidal functors
between monoidal categories, then an opmonoidal natural transformation
$\theta\colon P\nattrans Q$ is a natural transformation which commutes
with the opmonoidal structure transformations, in other words, the
following relations hold:
\begin{align*}
\raisebox{-.45\height}{\input{\figdir/H1_2_20.pstex_t}}&=\raisebox{-.45\height}{\input{\figdir/H1_2_21.pstex_t}}\mytag{OpmonNT1}\label{OpmonNT1};\\
\raisebox{-.45\height}{\begin{picture}(0,0)%
\includegraphics{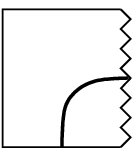}%
\end{picture}%
\setlength{\unitlength}{4144sp}%
\begingroup\makeatletter\ifx\SetFigFont\undefined%
\gdef\SetFigFont#1#2{%
  \fontsize{#1}{#2pt}%
  \selectfont}%
\fi\endgroup%
\begin{picture}(610,664)(979,-353)
\put(1191,-71){\makebox(0,0)[lb]{\smash{{\SetFigFont{8}{9.6}{\color[rgb]{0,0,0}$P$}%
}}}}
\end{picture}%
}&=\raisebox{-.45\height}{\input{\figdir/H1_2_22.pstex_t}}\mytag{OpmonNT2}\label{OpmonNT2}.
\end{align*}

\subsection{Opposite categories and contravariance}
Later on we will be interested in looking at duals in categories.
There is a slight problem from the point of view of diagrams in that a duality
\({}^\vee\) on a category \(\CC\) is actually a contravariant functor,
so is not a functor \(\CC\to \CC\) but rather can be considered a
covariant functor \(\CC^\op\to \CC\).  This means it will be extremely
convenient to be able to denote opposites in the diagrammatic
language.  First it is imperative to think about the operation of
taking opposites inside the two-category \(\CAT\).  Given a category
\(\CC\) the category \(\CC^\op\) is the category whose objects are in
canonical bijection with the objects of \(\CC\) but whose arrows are
reversed.  This means given a functor \(G\colon\CC\to \DD\) one
obtains a functor \(G^\op\colon\CC^\op \to \DD^\op\).  However, given
a natural transformation \(\theta\colon G_1\nattrans G_3\circ G_2\)
one obtains a natural transformation \(\theta^\op\colon G_3^\op\circ
G_2^\op\nattrans G_1^\op\).  This is an easy and informative exercise
for the reader.  Put another way, this gives a two-functor
\({}^\op\colon \CAT\to \CAT^\text{co}\) where \(\CAT^\text{co}\) denotes the
two-category of categories with the two-morphisms reversed.

In traditional notation we get a correspondence as follows:
  \[\raisebox{-.45\height}{\input{\figdir/H1_3_1.pstex_t}}\qquad\text{gives rise to}\qquad\raisebox{-.45\height}{\input{\figdir/H1_3_2.pstex_t}}.\]
To get this into the string pictures we will adopt the useful
convention that a shaded region means that it is the opposite
category, and the functors and natural transformations will be the
opposites of their labels.  Thus
  \[\raisebox{-.45\height}{\input{\figdir/H1_3_3.pstex_t}}\quad\text{gives rise to}\quad\raisebox{-.45\height}{\input{\figdir/H1_3_4.pstex_t}}
   \quad\text{denoted}\quad\raisebox{-.45\height}{\input{\figdir/H1_3_5.pstex_t}}.\]
Note the essential difference that ``shaded'' natural transformations are ``turned
upside-down''.  When we get to dinatural transformations, the shading
can be given the interpretation of the `other-side' of the surface.

This notational convention means that a contravariant functor
\({}^\vee\) on a category \(\CC\) can be denoted in the following way:
    \[\raisebox{-.45\height}{\input{\figdir/H1_3_6.pstex_t}}.\]
Note that the conditions needed to be satisfied by a duality on a
monoidal category will be stated in terms of dinatural
transformations, so we will not do that properly until later.

\section{Monads and bimonads}\label{Section:MonadsAndBimonads}
In this section we will look at monads and bimonads using string
diagrams.  In particular we consider the category of modules over a
monad from this perspective; this can be seen as a `formal' or
two-categorical point of view in that we discuss the category of
modules over a monad on a category without talking about the internal
structure of the category.  We go on to look at bimonads and how the
cateogry of modules in this case is monoidal, analogous to the
category of modules for a bialgebra.  Finally we will see how a
bimonad arises from a pair of adjoint functors.

\subsection{Monads}
This will be a quick diagrammatic recap on monads.  A monad on a
category \(\CC\) is an endofunctor \(T\colon \CC\to \CC\) together
with natural transformations \(\mu\colon T^2\nattrans T\) and
\(\iota\colon \Id_\CC\nattrans T\), known as the multiplication and unit and
drawn as
  \[\mu\equiv\raisebox{-.45\height}{\input{\figdir/H2_1_1.pstex_t}}\qquad\text{and}\qquad\iota\equiv\raisebox{-.45\height}{\begin{picture}(0,0)%
\includegraphics{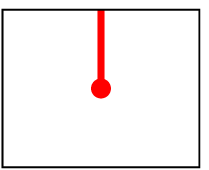}%
\end{picture}%
\setlength{\unitlength}{4144sp}%
\begingroup\makeatletter\ifx\SetFigFont\undefined%
\gdef\SetFigFont#1#2{%
  \fontsize{#1}{#2pt}%
  \selectfont}%
\fi\endgroup%
\begin{picture}(924,765)(169,-343)
\put(586,209){\rotatebox{360.0}{\makebox(0,0)[rb]{\smash{{\SetFigFont{8}{9.6}{\color[rgb]{0,0,0}$T$}%
}}}}}
\end{picture}%
}.\]
These have to satisfy the associativity and unit laws, namely
  \begin{gather}
    \raisebox{-.45\height}{\begin{picture}(0,0)%
\includegraphics{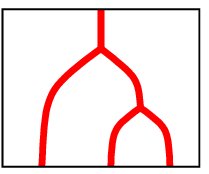}%
\end{picture}%
\setlength{\unitlength}{4144sp}%
\begingroup\makeatletter\ifx\SetFigFont\undefined%
\gdef\SetFigFont#1#2{%
  \fontsize{#1}{#2pt}%
  \selectfont}%
\fi\endgroup%
\begin{picture}(924,786)(169,-364)
\put(586,209){\rotatebox{360.0}{\makebox(0,0)[rb]{\smash{{\SetFigFont{8}{9.6}{\color[rgb]{0,0,0}$T$}%
}}}}}
\end{picture}%
}=\raisebox{-.45\height}{\begin{picture}(0,0)%
\includegraphics{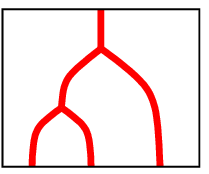}%
\end{picture}%
\setlength{\unitlength}{4144sp}%
\begingroup\makeatletter\ifx\SetFigFont\undefined%
\gdef\SetFigFont#1#2{%
  \fontsize{#1}{#2pt}%
  \selectfont}%
\fi\endgroup%
\begin{picture}(924,786)(1069,-364)
\put(1396,209){\makebox(0,0)[lb]{\smash{{\SetFigFont{8}{9.6}{\color[rgb]{0,0,0}$T$}%
}}}}
\end{picture}%
};\mytag{Monad1}\label{Monad1}\\
%    \qquad\text{and}\qquad
    \raisebox{-.45\height}{\input{\figdir/H2_1_5.pstex_t}}=\raisebox{-.45\height}{\begin{picture}(0,0)%
\includegraphics{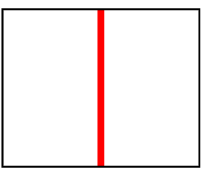}%
\end{picture}%
\setlength{\unitlength}{4144sp}%
\begingroup\makeatletter\ifx\SetFigFont\undefined%
\gdef\SetFigFont#1#2{%
  \fontsize{#1}{#2pt}%
  \selectfont}%
\fi\endgroup%
\begin{picture}(924,786)(1069,-364)
\put(1396,209){\makebox(0,0)[lb]{\smash{{\SetFigFont{8}{9.6}{\color[rgb]{0,0,0}$T$}%
}}}}
\end{picture}%
}=\raisebox{-.45\height}{\input{\figdir/H2_1_7.pstex_t}}.\mytag{Monad2}\label{Monad2}
  \end{gather}
There is an associated category \(T\CC\) of \emph{\(T\)-modules}, this
is sometimes written
\(\CC^T\).  The objects of this category
are pairs \(\bigl(m,\ \left(r\colon T(m)\to m\right)\bigr)\) where \(m\) is an object of
\(\CC\), such that the diagrams
\[\xymatrix{
T\circ T(m)\ar[r]^{Tr}
\ar[d]_{\mu_m}
&T(m)
\ar[d]^r
\\
T(m)
\ar[r]^r
&m
}
\quad\text{and}\quad \xymatrix{ m\ar[r]^{\iota_m} \ar[dr]_{\id_m}
  &T(m) \ar[d]^r
  \\
  &m}
\]
  % \[r\circ \mu_m=r\circ T(r)%\colon T^2(m)\to m\)
%    \quad\text{and}\quad r\circ \iota_m=\id_m%\colon m\to m
%   ,\]
both commute,
   as one would expect from anything befitting the name `module'.  The
   morphisms in the category of \(T\)-modules are morphisms between
   the underlying objects of \(\CC\) that commute with the
   \(T\)-action.  The example that I have in mind here is where
   \(\CC\) is a monoidal category, \(A\) is a unital algebra object in
   \(\CC\) and \(T\) is the monad \(A\otimes{-}\); in this case the
   category of \(T\)-modules is precisely the category of
   \(A\)-modules in the usual sense.

We fit the category of modules into the graphical calculus by
considering associated functors and natural transformations.  An
object of the
category of modules consists of an object of \(\CC\)  and an action
morphisms, an alternative view of such a pair is that we have a
forgetful functor \(U_T\colon T\CC\to \CC\) given
   by \(U_T(m,r):=m\), which just forgets the action, and we have a
   natural transformation \(\rho\colon T\circ U_T\nattrans U_T\)
   defined by \(\rho_{(m,r)}:=r\) which encodes the action.  We will
   denote the forgetful functor $U_T$ by a dashed-dotted line and draw
 the natural transformation \(\rho\) as follows.
\[\rho\equiv\raisebox{-.45\height}{\input{\figdir/H2_1_21.pstex_t}}\]
The module conditions above become the following:
  \begin{gather}
    \raisebox{-.45\height}{\input{\figdir/H2_1_22.pstex_t}}=\raisebox{-.45\height}{\input{\figdir/H2_1_23.pstex_t}};\mytag{Module1}\label{module1}\\
%    \qquad\text{and}\qquad
    \raisebox{-.45\height}{\input{\figdir/H2_1_24.pstex_t}}=\raisebox{-.45\height}{\begin{picture}(0,0)%
\includegraphics{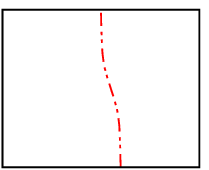}%
\end{picture}%
\setlength{\unitlength}{4144sp}%
\begingroup\makeatletter\ifx\SetFigFont\undefined%
\gdef\SetFigFont#1#2{%
  \fontsize{#1}{#2pt}%
  \selectfont}%
\fi\endgroup%
\begin{picture}(924,744)(1069,-343)
\put(1644,-240){\makebox(0,0)[lb]{\smash{{\SetFigFont{8}{9.6}{\color[rgb]{0,0,0}$U_T$}%
}}}}
\end{picture}%
}.\mytag{Module2}\label{module2}
  \end{gather}

We also have a
free module functor \(F_T\colon \CC\to T\CC\).  This is given on objects by
%\begin{align*}
%  U_T(m,r)&:=m;\\
 \[ F_T(x):= \left(T(x),\ \left(\mu_{T(x)}\colon T(T(x))\to T(x)\right)\right).\]
%\end{align*}
Note that \(T\) is precisely the composite \(U_T\circ F_T\),
so we have identity natural transformations $\Id\colon U_T\circ
F_T\stackrel\sim\nattrans T$ and $\Id \colon T\stackrel\sim\nattrans
U_T\circ F_T$ which we draw as
 \[\raisebox{-.45\height}{\input{\figdir/H2_1_8.pstex_t}}\qquad\text{and}\qquad\raisebox{-.45\height}{\input{\figdir/H2_1_9.pstex_t}}.\]
Here it should be observed that the graphical language fails to
distinguish between identity natural transformations and natural
isomorphisms.

Note that the multiplication $\mu$ on $T$ is recovered from these
identifications together with the action natural transformation $\rho$
in the following way:
  \[\raisebox{-.45\height}{\input{\figdir/H2_1_1.pstex_t}}=\raisebox{-.45\height}{\input{\figdir/H2_1_11.pstex_t}}.\]

\subsection{Monads from adjoint functors}
A standard way of obtaining monads is via pairs of adjoint functors.
Suppose that $F$ and $U$ form such a pair, \(F\dashv U\), then
\(U\circ F\colon \CC\to\CC\) forms a monad.  The multiplication and
unit of the monad are obtained from the unit and counit of the
adjunction in the following easily drawn fashion:
  \[\mu\equiv\raisebox{-.45\height}{\input{\figdir/H2_2_1.pstex_t}};\qquad \iota\equiv\raisebox{-.45\height}{\input{\figdir/H2_2_2.pstex_t}}.\]
If the reader has not seen this before then they should immediately
verify diagrammatically that the axioms of a monad are satisfied.

It should be noted that every monad \(T\) actually arises in this way,
as the composite of a left and a right adjoint; for example,
there is an adjunction \(F_T\dashv U_T\) between the free and
forgetful functors described above.  In general there will be
several different adjoint decompositions of a monad.

\subsection{Bimonads}\label{section:bimonads}
We now bring monads and monoidal categories together.  Suppose that
\(T\colon \CC\to\CC\) is a monad on a monoidal category
\((\CC,\otimes,\one)\).  We can then ask the question ``Under what
circumstances does the monoidal structure on \(\CC\) lift to a monoidal structure on the category of \(T\)-modules \(T\CC\)?''  Or, we could ask the weaker and less precise question ``Given two \(T\)-modules \(\left(m,\left(r\colon T(m)\to m\right)\right)\) and
\(\left(m',\left(r'\colon T(m')\to m'\right)\right)\) how do we obtain a natural \(T\)-module structure
on \(m\otimes m'\)?''

The answer to the first question (and hence the second) was found by Moerdijk, but before stating the answer we should introduce the following piece of terminology.
\begin{defn} A monad \(T\colon\CC\to \CC\) on a monoidal category has
  a \emph{bimonad} (or \emph{opmonoidal monad}) structure if the
  functor \(T\) has an opmonoidal structure with respect to which both
  the multiplication
  \(\mu\) and the unit \(\iota\) are opmonoidal natural
  transformations (in the sense of Section~\ref{Section:OpMonNT}).
\end{defn}
Bimonads are so named because of the analogy with bialgebras given by
Theorem~\ref{Thm:Moerdijk} below, though they were called Hopf monads
in his original paper.  Before stating the Theorem it is worth
unpacking this rather concise definition a
little.  A monad has a bimonad structure if firstly there is an
opmonoidal structure for $T$, that is there are specified
natural transformation
\[\sigma_2^T\colon \otimes\circ (T\times
T)\nattrans T\circ \otimes
\quad\text{and} \quad
\sigma_0^T\colon T\circ \one \nattrans \one,\]
 drawn as follows, in which hopefully it is clear where the hidden
 lines go.
\[\sigma_2^T\equiv\raisebox{-.45\height}{\begin{picture}(0,0)%
\includegraphics{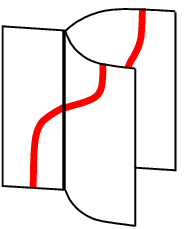}%
\end{picture}%
\setlength{\unitlength}{4144sp}%
\begingroup\makeatletter\ifx\SetFigFont\undefined%
\gdef\SetFigFont#1#2{%
  \fontsize{#1}{#2pt}%
  \selectfont}%
\fi\endgroup%
\begin{picture}(814,1040)(1630,-4547)
\end{picture}%
}, \qquad
\sigma_0^T\equiv\raisebox{-.45\height}{\begin{picture}(0,0)%
\includegraphics{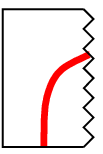}%
\end{picture}%
\setlength{\unitlength}{4144sp}%
\begingroup\makeatletter\ifx\SetFigFont\undefined%
\gdef\SetFigFont#1#2{%
  \fontsize{#1}{#2pt}%
  \selectfont}%
\fi\endgroup%
\begin{picture}(440,674)(2932,-4354)
\end{picture}%
}\] These natural transformations must
satisfy the axioms (\textsf{Opmondl1--3}) and the multiplication and
unit must be opmonoidal with respect to this which means  that
the following must hold:
  \[\raisebox{-.45\height}{\begin{picture}(0,0)%
\includegraphics{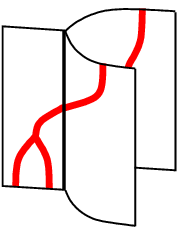}%
\end{picture}%
\setlength{\unitlength}{4144sp}%
\begingroup\makeatletter\ifx\SetFigFont\undefined%
\gdef\SetFigFont#1#2{%
  \fontsize{#1}{#2pt}%
  \selectfont}%
\fi\endgroup%
\begin{picture}(814,1040)(1630,-4547)
\end{picture}%
}=\raisebox{-.45\height}{\begin{picture}(0,0)%
\includegraphics{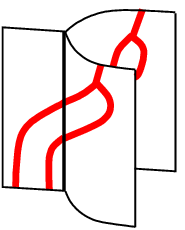}%
\end{picture}%
\setlength{\unitlength}{4144sp}%
\begingroup\makeatletter\ifx\SetFigFont\undefined%
\gdef\SetFigFont#1#2{%
  \fontsize{#1}{#2pt}%
  \selectfont}%
\fi\endgroup%
\begin{picture}(814,1040)(1630,-4547)
\end{picture}%
};
  \qquad
  %\text{and}\qquad
  \raisebox{-.45\height}{\begin{picture}(0,0)%
\includegraphics{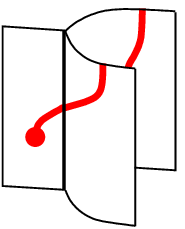}%
\end{picture}%
\setlength{\unitlength}{4144sp}%
\begingroup\makeatletter\ifx\SetFigFont\undefined%
\gdef\SetFigFont#1#2{%
  \fontsize{#1}{#2pt}%
  \selectfont}%
\fi\endgroup%
\begin{picture}(814,1040)(1630,-4547)
\end{picture}%
}=\raisebox{-.45\height}{\begin{picture}(0,0)%
\includegraphics{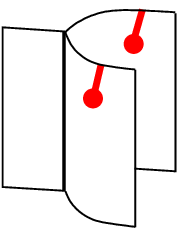}%
\end{picture}%
\setlength{\unitlength}{4144sp}%
\begingroup\makeatletter\ifx\SetFigFont\undefined%
\gdef\SetFigFont#1#2{%
  \fontsize{#1}{#2pt}%
  \selectfont}%
\fi\endgroup%
\begin{picture}(814,1040)(1630,-4547)
\end{picture}%
}\mytag{BM1--2}\label{BM1};
  \]
\[\raisebox{-.45\height}{\begin{picture}(0,0)%
\includegraphics{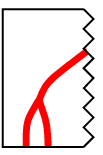}%
\end{picture}%
\setlength{\unitlength}{4144sp}%
\begingroup\makeatletter\ifx\SetFigFont\undefined%
\gdef\SetFigFont#1#2{%
  \fontsize{#1}{#2pt}%
  \selectfont}%
\fi\endgroup%
\begin{picture}(439,678)(2932,-4358)
\end{picture}%
}=\raisebox{-.45\height}{\begin{picture}(0,0)%
\includegraphics{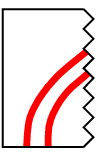}%
\end{picture}%
\setlength{\unitlength}{4144sp}%
\begingroup\makeatletter\ifx\SetFigFont\undefined%
\gdef\SetFigFont#1#2{%
  \fontsize{#1}{#2pt}%
  \selectfont}%
\fi\endgroup%
\begin{picture}(439,678)(2932,-4358)
\end{picture}%
}; \qquad \raisebox{-.45\height}{\begin{picture}(0,0)%
\includegraphics{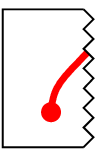}%
\end{picture}%
\setlength{\unitlength}{4144sp}%
\begingroup\makeatletter\ifx\SetFigFont\undefined%
\gdef\SetFigFont#1#2{%
  \fontsize{#1}{#2pt}%
  \selectfont}%
\fi\endgroup%
\begin{picture}(439,654)(2932,-4334)
\end{picture}%
}=\raisebox{-.45\height}{\begin{picture}(0,0)%
\includegraphics{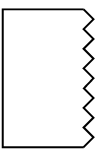}%
\end{picture}%
\setlength{\unitlength}{4144sp}%
\begingroup\makeatletter\ifx\SetFigFont\undefined%
\gdef\SetFigFont#1#2{%
  \fontsize{#1}{#2pt}%
  \selectfont}%
\fi\endgroup%
\begin{picture}(439,654)(2932,-4334)
\end{picture}%
}.
\mytag{BM3--4}\label{BM2}\]
The result of Moerdijk can now be stated.
\begin{thm}[Moerdijk \cite{Moerdijk:MonadsTensorCategories}]
  \label{Thm:Moerdijk}
  Let \(T\) be a monad on a monoidal category \(\CC\) and let \(T\CC\)
  be its category of modules.  Then specifying a
  lift of the monoidal structure on \(\CC\) to a monoidal structure on
  \(T\CC\) is precisely the same as
  specifying a bimonad structure on \(T\).
\end{thm}
We will spend the rest of this section seeing why this is true.  To do
this, the following is an extremely useful observation.
%
%To lift the tensor product \({\otimes}\colon \CC\times\CC\to \CC\) to
%a tensor product \({\otimes}^T\colon T\CC\times T\CC\to T\CC\) the
%following is a very useful observation:
\begin{prop}\label{Lemma:LiftingToModules}
  For \(T\colon \CC\to \CC\) a monad, \(T\CC\) its category of modules
  and \(\DD\) any category, specifying a functor \(H\colon \DD\to
  T\CC\) is precisely the same as specifying a functor
  \(\pre{U}H\colon \DD\to \CC\) and a natural transformation
  \(\rho_H\colon T\circ \pre{U}H\nattrans \pre{U}H\), drawn as
  \[\raisebox{-.45\height}{\input{\figdir/H2_3_1.pstex_t}},\]
  such that the following two relations are satisfied:
  \[
  \raisebox{-.45\height}{\input{\figdir/H2_3_2.pstex_t}}=\raisebox{-.45\height}{\input{\figdir/H2_3_3.pstex_t}}
  \quad\text{and}\quad
  \raisebox{-.45\height}{\input{\figdir/H2_3_4.pstex_t}}=\raisebox{-.45\height}{\begin{picture}(0,0)%
\includegraphics{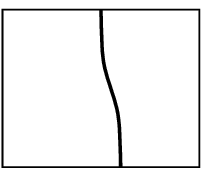}%
\end{picture}%
\setlength{\unitlength}{4144sp}%
\begingroup\makeatletter\ifx\SetFigFont\undefined%
\gdef\SetFigFont#1#2{%
  \fontsize{#1}{#2pt}%
  \selectfont}%
\fi\endgroup%
\begin{picture}(924,764)(1069,-353)
\put(1644,-240){\makebox(0,0)[lb]{\smash{{\SetFigFont{8}{9.6}{\color[rgb]{0,0,0}$\pre{U}H$}%
}}}}
\end{picture}%
}.
  \mytag{R1--2}\label{R12}
  \]
\end{prop}
\begin{proof}[Proof (sketch)]
  This is just the fact that \(T\CC\) consists of pairs \((m,r)\),
  where \(m\in \CC\) and \(r\colon T(m)\to m\) is an action.  So given
  such a functor \(H\) define \(\pre{U}H\) to be \(U_T\circ H\) and
  \(\rho^H\colon T\circ U_T\circ H\nattrans U_T\circ H\) to be \(\rho\circ
  \id_H\).
  Conversely, given such a pair, define \(H(d):=\left(\pre{U}H(d),\ \left(
  \rho^H_d\colon T(\pre{U}H(d))\to \pre{U}H(d)\right)\right)\).
\end{proof}

Now we will see how to prove Moerdijk's Theorem.  To lift the monoidal
structure of \(\CC\) to a monoidal structure on \(T\CC\) we need to
specify the tensor product and unit on \(T\CC\), together with the
associativity and unital natural transformations.  We will concentrate
on the tensor product.

For  a functor \({\otimes^T}\colon \CC\times T\CC\to T\CC\) to be a
lift of a  functor \({\otimes}\colon \CC\times \CC\to \CC\) (with
respect to the forgetful functor \(U_T\colon T\CC \to \CC\))  the
following diagram must commute.
\[\xymatrix{
T\CC\times T\CC
\ar[d]_{U_T\times U_T}
\ar[r]^{\otimes^T}
&
T\CC
\ar[d]^{U_T}
\\
\CC\times\CC
\ar[r]^{\otimes}
&\CC
}
\]
Algebraically this means that
\(U_T\circ \otimes^T = {\otimes}\circ (U_T\times U_T)\), so in the
notation of Proposition~\ref{Lemma:LiftingToModules} we have \(\pre{U}\otimes^T{}={\otimes}\circ (U_T\times
U_T)\). By the proposition then, lifting the functor \({\otimes}\) is precisely the same as specifying
a natural transformation \[\rho_{\otimes}\colon T\circ {\otimes}\circ (U_T\times
U_T)\nattrans {\otimes}\circ (U_T \times U_T)\]
which satisfies the two conditions of the lemma.
(Intuitively this says that to specify a lifted tensor product of two
modules we need to specify an action of $T$ on the tensor product of
the two objects underlying the modules.)
We draw the natural transformation as
  \[\rho_\otimes\equiv\raisebox{-.45\height}{\begin{picture}(0,0)%
\includegraphics{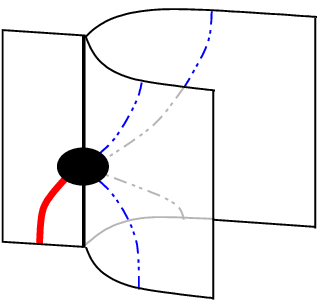}%
\end{picture}%
\setlength{\unitlength}{4144sp}%
\begingroup\makeatletter\ifx\SetFigFont\undefined%
\gdef\SetFigFont#1#2{%
  \fontsize{#1}{#2pt}%
  \selectfont}%
\fi\endgroup%
\begin{picture}(1461,1351)(1630,-4895)
\end{picture}%
},\]
In general for an adjunction \(F\dashv U\),  and with $G$ and $H$
functors with suitable source and target, there is a bijection between
sets of natural transformations:
  \[
  \Nat(G,H\circ F)
  \cong \Nat(G\circ U, H).
  \]
Hence, because there is the adjunction \(F_T\dashv
U_T\) and the monad factorizes as \(T=U_T\circ F_T\), for any functors $G,H\colon \CC\to\DD$ there are the following
identifications of sets of natural transformations
  \[
  \Nat(G,H\circ T)
  =\Nat(G, H\circ U_T\circ F_T)
  \cong \Nat(G\circ U_T, H\circ U_T).
  \]
In the diagrammatic notation, the isomorphism between the two outside sets is given by
 \[\raisebox{-.45\height}{\input{\figdir/H2_3_7.pstex_t}}\mapsto \raisebox{-.45\height}{\input{\figdir/H2_3_8.pstex_t}}\qquad\text{and}\qquad
  \raisebox{-.45\height}{\input{\figdir/H2_3_9.pstex_t}}\mapsto \raisebox{-.45\height}{\input{\figdir/H2_3_10.pstex_t}}.\]
Similarly there is an identification
  \[
  \Nat(T\circ{\otimes},{\otimes}\circ(T\times T))
  \cong
  \Nat(T\circ{\otimes}\circ(U_T\times U_T),{\otimes}\circ(U_T\times U_T))
  \]
so the natural transformation \(\rho_{\otimes}\) above is equivalent to a natural transformation
  \[\raisebox{-.45\height}{}\equiv\raisebox{-.45\height}{\begin{picture}(0,0)%
\includegraphics{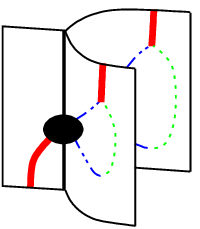}%
\end{picture}%
\setlength{\unitlength}{4144sp}%
\begingroup\makeatletter\ifx\SetFigFont\undefined%
\gdef\SetFigFont#1#2{%
  \fontsize{#1}{#2pt}%
  \selectfont}%
\fi\endgroup%
\begin{picture}(883,1035)(1630,-4547)
\end{picture}%
}\]
satisfying the two conditions \bref{R12} of the proposition, but these
are clearly axioms \bref{BM1}.
%  \[\pstex{H2_3_13}=\pstex{H2_3_14}
%  \qquad\text{and}\qquad
%  \pstex{H2_3_15}=\pstex{H2_3_16}.
%  \]
(Thus we have a natural transformation \(T\circ \otimes \nattrans
\otimes \circ(T\times T)\) which is used to write down an action on
the tensor product, as if we have $T$-modules \((m,r)\) and
\((m',r')\) then we will define the action on \(m\otimes m'\) by
\(T(m\otimes m')\to Tm\otimes Tm'\xleftarrow{r\otimes r'} m\otimes
m'.\))

Having lifted the tensor product we also need to lift the unit
$\one\colon \pnt\to \CC$; this gives the morphism $\sigma_0^T$
satisfying \handtag{BM3--4}.  Now for these two functors to combine to
give a tensor product structure they must give an opmonoidal structure
for $T$.  This gives Theorem~\ref{Thm:Moerdijk}.  Full details are
found in \cite{Moerdijk:MonadsTensorCategories}.

%% This means that the least that we require is a
%% natural transformation
%%  \[T\circ \otimes \circ(U_T\times U_T)\nattrans \otimes \circ(U_T\times
%%  U_T),\]
%% which we denote as follows:
%%  \[\pichere.\]
%% But we also have the following identification, for any functors
%% \(G,H\colon\CC\to \DD\), using the adjunction
%% \(F_T\dashv U_T\) and the identity \(T=U_TF_T\):
%%  \[\Nat(G,H\circ T)\cong \Nat(G\circ U_T, H\circ U_T).\]
%% Explicitly, the maps are given by
%%  \[\pichere\mapsto \pichere\qquad\text{and}\qquad
%%   \pichere\mapsto \pichere.\]
%% So by the obvious generalization of this result, specifying the natural
%% transformation above is equivalent to specifying a natural
%% transformation
%%  \[T\circ \otimes \nattrans \otimes\circ (T\times T),\]
%% w

\subsection{Bimonads from adjoint functors}
\label{Subsection:BimonadsFromAdjoint}
Suppose now that \(U\colon \DD\to \CC \) is a strong monoidal functor
between monoidal categories, so we have the following pairs of inverse natural transformations:
 \[\raisebox{-.45\height}{\input{\figdir/H2_4_1.pstex_t}}\qquad\text{and}\qquad\raisebox{-.45\height}{\input{\figdir/H2_4_2.pstex_t}};\]
 \[\raisebox{-.45\height}{\begin{picture}(0,0)%
\includegraphics{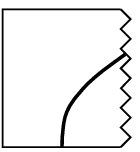}%
\end{picture}%
\setlength{\unitlength}{4144sp}%
\begingroup\makeatletter\ifx\SetFigFont\undefined%
\gdef\SetFigFont#1#2{%
  \fontsize{#1}{#2pt}%
  \selectfont}%
\fi\endgroup%
\begin{picture}(609,664)(979,-353)
\put(1313,-74){\makebox(0,0)[rb]{\smash{{\SetFigFont{8}{9.6}{\color[rgb]{0,0,0}$U$}%
}}}}
\end{picture}%
}\qquad\text{and}\qquad\raisebox{-.45\height}{\begin{picture}(0,0)%
\includegraphics{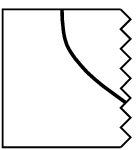}%
\end{picture}%
\setlength{\unitlength}{4144sp}%
\begingroup\makeatletter\ifx\SetFigFont\undefined%
\gdef\SetFigFont#1#2{%
  \fontsize{#1}{#2pt}%
  \selectfont}%
\fi\endgroup%
\begin{picture}(609,664)(979,-343)
\put(1313,-74){\makebox(0,0)[rb]{\smash{{\SetFigFont{8}{9.6}{\color[rgb]{0,0,0}$U$}%
}}}}
\end{picture}%
}.\]
And suppose that \(U\) has a left adjoint \(F\colon\CC\to\DD\) then, as
mentioned in Section~\ref{Subsection:MonoidalFunctors}, \(F\) is opmonoidal; and as \(U\) is also opmonoidal,
so the composite \(U\circ F\) is opmonoidal, with the opmonoidal
structure given explicitly as follows:
 \[\sigma^{U\circ F}_2\equiv\raisebox{-.45\height}{\input{\figdir/H2_4_3.pstex_t}}; \qquad \sigma^{U\circ F}_0\equiv\raisebox{-.45\height}{\input{\figdir/H2_4_13.pstex_t}}.\]
Of course, \(U\circ F\) is also a monad, and it is very easy to see in
this pictorial language that it is an opmonoidal monad, ie.\ that the
product and unit of the monad are opmonoidal transformations.

For instance, the following proves that \handtag{BM1} holds:
 \begin{align*}
   \raisebox{-.45\height}{\input{\figdir/H2_4_4.pstex_t}}
   &=\raisebox{-.45\height}{\input{\figdir/H2_4_5.pstex_t}}=\raisebox{-.45\height}{\input{\figdir/H2_4_6.pstex_t}}\\
   &=\raisebox{-.45\height}{\input{\figdir/H2_4_7.pstex_t}}.
 \end{align*}
%and
% \[\pstex{H2_4_8}=\pstex{H2_4_9}=\pstex{H2_4_10}.\]
Thus in the case that \(U\) is strongly monoidal, \(U\circ F\) is a
bimonad.

Note that if \(T\) is a bimonad, then the category \(T\CC\) of
\(T\)-modules is monoidal and in the decomposition \(T=U_T\circ F_T\)
the forgetful functor \(U_T\) is strongly monoidal.

\section{The diagrammatics of dinatural transformations}
In this section we introduce dinatural transformations so that we can
give a diagrammatic description of duality on a monoidal category, the
point being that evaluation and coevaluation are dinatural rather than
natural transformations.

\subsection{Motivating example: a monoidal category with duals}
The first question to address is ``What is an appropriate notion of a
monoidal category with duals?''  To simplify the situation, we
will just consider \emph{left} duals: right duals can be handled
similarly.  So suppose that \((\CC,\otimes,\one)\) is a monoidal
category, a left duality on \(\CC\) will be a functor \({}^\vee\colon
\CC^\op\to \CC\) together with evaluation and coevaluation maps for
every object \(a\) in the category,
\[\ev_a\colon \prevee a\otimes a\to \one
 \quad\text{and}\quad
 \coev_a\colon \one \to a\otimes \prevee a,\]
such that for every morphism \(f\colon a\to a'\) in the category
the following naturality conditions hold:
% \begin{align*}
\[
 \ev_a\circ (\prevee f\otimes \id)%&
   =\ev_{a'}\circ ( \id\otimes f)%\\
 \quad\text{and}\quad
 (f\otimes \id)\circ \coev_{a}%&
   =( \id\otimes \prevee f)\circ\coev_{a'},
\]
%\end{align*}
and such that the following ``snake'' relations hold,
 \begin{align*}(\id_a\otimes \ev_a)\circ (\coev_a\otimes \id_a)&=\id_a\\
% \qquad\text{and}\qquad
 (\ev_a\otimes \id_{\prevee a}) \circ  (\id_{\prevee a}\otimes \coev_a)
 & =\id_{\prevee a.}
 \end{align*}
%% This last condition is drawn as a commutative diagram as follows:
%%  \[\xymatrix{
%%   &\prevee a \otimes a\ar[dr]^{\ev_a}\\
%%    \prevee a'\otimes a\ar[ur]^{\prevee f\otimes \id}
%%     \ar[dr]^{\id\otimes f}
%%    &&\quad\one\quad.\\
%%    &\prevee a' \otimes a'\ar[ur]^{\ev_{a'}}
%%   }\]

One would like to interpret \(\ev\) and \(\coev\) as some sort of natural
transformations, so that, for instance, \(\ev\) would be a natural
transformation from the ``functor'' \(\CC\to\CC\) given by
\(a\mapsto \prevee a\otimes
a\) to the functor \(\CC\to\CC\) given by \(a\mapsto \one\): however the former is \emph{not} a functor.  So we
consider the functor \(\CC^\op\times \CC\to \CC;\ (b,a)\mapsto
\prevee b \otimes a\) and the functor \(\one\colon \pnt\to \CC\).
% A dinatural transformation \(\alpha\colon P\dinat \one\) will then
% be
Evaluation is then
a family of morphisms \(\ev_a\colon \prevee a\otimes a\to \one\), indexed by the
objects of \(\CC\), such that the following diagrams commutes:
 \[\xymatrix{
  &\prevee a\otimes a\ar[dr]^{\ev_a}\\
   \prevee a'\otimes a\ar[ur]^{\prevee f\otimes \id_a}
    \ar[dr]_{\id_{a'}\otimes f }
   &&\quad\one\quad.\\
   &\prevee a'\otimes a'\ar[ur]_{\ev_{a'}}
  }\]

Now, because, amongst other reasons, we also want to deal with
coevaluation \(\coev_a\colon \one \to a\otimes \prevee a\), and the
constraint \((\id_a\otimes \ev_a)\circ (\coev_a\otimes \id_a)=\id_a\),
we introduce the more general notion of dinatural transformation.

\subsection{Dinatural transformations}
Motivated by the above example, we are led to Eilenberg and Kelly's
notion of dinatural transformation.
Suppose
\[P\colon \CC\times \CC^\op\times \AA\to \BB
 \quad\text{and}\quad
 Q\colon \AA\times \DD^\op\times \DD\to \BB\]
are two functors, then a \emph{dinatural transformation} \(\beta\colon
 P\dinat Q\) is a family of morphisms
\(\beta_{c,a,d}\colon P(c,c,a)\to Q(a,d,d)\) which satisfies
 the following naturality condition.  If $f\colon a\to a'$, $g\colon c\to c'$ and $h\colon d \to d'$ are morphisms in \(\AA\), \(\CC\) and \(\DD\) respectively then the diagram below commutes.
\[\xymatrix{
  &P(c,c,a')\ar[r]^{\beta_{c,a',d}}&
  Q(a',d,d)\ar[dr]^{Q(\id_{a'},\id_d,h)}\\
  P(c,c',a)\ar[ur]^{P(\id_c,g,f)}\ar[dr]_{P(g,\id_c,\id_a)}&&&
  Q(a',d,d')\\
  &P(c',c',a)\ar[r]^{\beta_{c',a,d'}}&
  Q(a,d',d')\ar[ur]_{Q(f,h,\id_{d'})}
  }\]

Note that each of the categories \(\AA\), \(\CC\), \(\DD\) and \(\DD\) can be products of
other categories or can indeed be the terminal category \(\pnt\) and can
all be permuted in the definition.  Thus the
case of evaluation described above occurs when \(\CC=\BB\) and
\(\AA=\DD=\pnt\), and a usual natural transformation is the case where
\(\CC=\DD=\pnt\).

We can then denote such a dinatural transformation as follows.  For functors
\[P\colon \CC\times \CC^\op\times \AA\to \BB
 \quad\text{and}\quad
 Q\colon \AA\times \DD^\op\times \DD\to \BB\]
a dinatural transformation \(\beta\colon
 P\dinat Q\) is denoted
 \[\raisebox{-.45\height}{\input{\figdir/H3_2_1.pstex_t}},\]
where as usual the diagram is read upwards.

The first thing to note is that in the case of a natural
transformation, when \(\CC=\DD=\pnt\), we recover the usual string
diagram notation.

The next thing to note is the right-hand profile of the surface.  This
is the so-called Eilenberg-Kelly graph of the dinatural
transformation, consisting of arcs with the end-points of an arc
labelled by the same category.  These graphs are important in the
composition of dinatural transformations as we will see below.

It can also be pointed out that a dinatural transformation can
actually be written as a natural transformation.  For example a
dinatural transformation \(\beta\colon P\dinat Q\) as above is
equivalent to a natural transformation
\[  \Hom_\CC(-,-) \times\Hom_\AA(-,-)\times
\Hom_\DD(-,-) \nattrans \Hom_\BB\left(P({-},{-},{-}),Q({-},{-},{-})\right)\]
between functors from \(\CC^\op\times \CC\times\AA^\op\times \AA\times
\DD^\op\times \DD\) to \(\Set\), so this also allows a diagrammatic
description, but it is rather messier and the composition is not as
straightforward.

\subsection{Vertical composition}
In order to make sense of the snake condition on \(\ev\) and \(\coev\)
we will need to define the vertical composition of dinatural
transformations.
If \(\beta'\colon P\dinat Q\) and \(\beta\colon Q\dinat
R\) are two dinatural transformations then the composite
\(\beta\circ\beta'\) can not always be defined; however the composite
can be defined,
resulting in a dinatural transformation, in the case that the composite
of the Eilenberg-Kelly graphs contains no loops, as I will now
explain.

This is best illustrated by our motivating example.
If we have a monoidal category \(\CC\) with a functor \(\prevee\colon
\CC^\op\to \CC\) and dinatural transformations
\(\ev\colon\prevee\otimes \Id_\CC\dinat \one\) and
\(\coev\colon\one\dinat \Id_\CC\otimes\prevee\), where, for instance \(\prevee\otimes \Id_\CC\) means \({\otimes}\circ(\Id_\CC\times \prevee)\). Then we can draw these dinatural transformations as
 \[\ev\equiv \raisebox{-.45\height}{\begin{picture}(0,0)%
\includegraphics{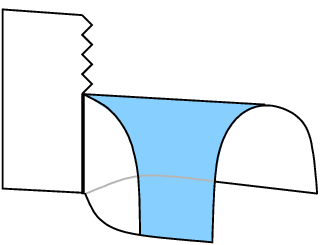}%
\end{picture}%
\setlength{\unitlength}{4144sp}%
\begingroup\makeatletter\ifx\SetFigFont\undefined%
\gdef\SetFigFont#1#2{%
  \fontsize{#1}{#2pt}%
  \selectfont}%
\fi\endgroup%
\begin{picture}(1464,1089)(2679,-3068)
\end{picture}%
}\qquad \coev\equiv \raisebox{-.45\height}{\begin{picture}(0,0)%
\includegraphics{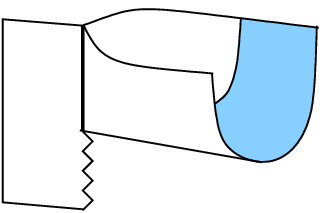}%
\end{picture}%
\setlength{\unitlength}{4144sp}%
\begingroup\makeatletter\ifx\SetFigFont\undefined%
\gdef\SetFigFont#1#2{%
  \fontsize{#1}{#2pt}%
  \selectfont}%
\fi\endgroup%
\begin{picture}(1468,943)(3391,-3353)
\end{picture}%
}.\]
We then have a dinatural transformation
\(\id\otimes\ev\colon\Id_\CC\otimes (\prevee\otimes \Id_\CC)\dinat
\Id_\CC\) whose components are given by \( (\id\otimes
\ev)_{a'',a}:=\id_{a''}\otimes\ev_a\colon a''\otimes(\prevee a \otimes
a)\to a''\). This should be drawn with binary tensor products,
but, by identifying \({\otimes}\circ(\otimes\times \Id_\CC)\) with a
triple tensor product \({-}\otimes{-}\otimes{-}\), it will be drawn as
\[\id\otimes\ev\equiv\raisebox{-.45\height}{\begin{picture}(0,0)%
\includegraphics{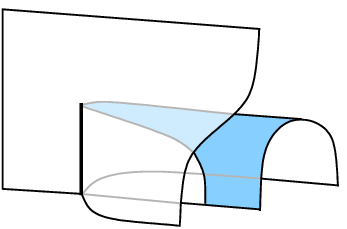}%
\end{picture}%
\setlength{\unitlength}{4144sp}%
\begingroup\makeatletter\ifx\SetFigFont\undefined%
\gdef\SetFigFont#1#2{%
  \fontsize{#1}{#2pt}%
  \selectfont}%
\fi\endgroup%
\begin{picture}(1557,1017)(2681,-2867)
\end{picture}%
}.\]
Similarly we have \(\coev\otimes\id\colon \Id_\CC\dinat (\Id_\CC\otimes \prevee)\otimes \Id_\CC\) with its components given by \([\coev\otimes\id]_{a',a}:=\coev_{a'}\otimes \id_a\colon a\to (a'\otimes \prevee (a'))\otimes a\) which is drawn as
 \[\coev\otimes\id\equiv\raisebox{-.45\height}{\begin{picture}(0,0)%
\includegraphics{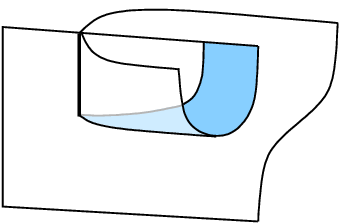}%
\end{picture}%
\setlength{\unitlength}{4144sp}%
\begingroup\makeatletter\ifx\SetFigFont\undefined%
\gdef\SetFigFont#1#2{%
  \fontsize{#1}{#2pt}%
  \selectfont}%
\fi\endgroup%
\begin{picture}(1557,992)(2686,-3589)
\end{picture}%
}.\]
The  two  dinatural transformations pictured above can be vertically composed to give a dinatural transformation, which is actually a natural transformation, \(\id\otimes \ev)\circ(\coev\otimes \id)\colon \Id_\CC\nattrans\Id_\CC\)
given by \[[(\id\otimes\ev)\circ(\coev\otimes \id)]_a:=
  (\id\otimes\ev)_{a,a}\circ(\coev\otimes \id)]_{a,a}\colon
   a\to a\]
and drawn as
\[\raisebox{-.45\height}{\begin{picture}(0,0)%
\includegraphics{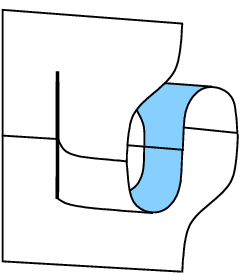}%
\end{picture}%
\setlength{\unitlength}{4144sp}%
\begingroup\makeatletter\ifx\SetFigFont\undefined%
\gdef\SetFigFont#1#2{%
  \fontsize{#1}{#2pt}%
  \selectfont}%
\fi\endgroup%
\begin{picture}(1099,1224)(2681,-3074)
\end{picture}%
}.\]
One of the conditions for \(({}^\vee, \ev, \coev)\) to form a duality on the monoidal category $\CC$ is that the above should be the identity natural transformation.  The two conditions are drawn in the next subsection, below.
%
%Define \(P\colon \CC\times
%\CC^\op \times \CC\to \CC\) by \(P(a,a',a''):=a\otimes \prevee a'\otimes
%a''\), and \(Q\colon \CC\to\CC\) by \(Q(a):=a\).  Then we have a
%dinatural transformation \((\id\otimes \ev)\colon P\dinat Q\) with
%components
% \[(\id\times \ev)_{a,a'}:=\id_a\otimes \ev_{a'} \colon
%    a\otimes \prevee a'\otimes a'\to a\]
%drawn as
%
%and we have \((\coev\times \id)\colon Q\dinat P\) with components
%\[(\coev\otimes \id)_{a',a}:=\coev_{a'}\otimes \id_a\colon a\to a'\otimes
%\prevee a' \otimes a,\] similarly drawn as
%
%Now the composition
%
%actually gives a natural transformation \(Q\nattrans Q\); indeed one
%   condition that we have a duality is that this composition is the
%   identity natural transformation.  We can draw this condition then
%   as
%\[\pstex{H3_2_4}=\pstex{H3_2_5}.\]
%The other condition on the duality is that
%\((\ev\otimes\id)\circ(\id\otimes \coev)\) is the identity natural
%transformation on the dual; pictorially we have
% \[\pstex{H3_3_5}=\pstex{H3_3_6}.\]

More generally, we can define the composite of two dinatural
transformations provided the composite Eilenberg-Kelly graph has no
loops.  For example we can form a dinatural transformation from two
dinatural transformations of the following form.  Suppose we have
functors
\[P\colon\pt\to \BB,\quad
R\colon \CC\times \CC^\op\to \BB,\quad
Q\colon \CC\times\CC^\op\times\CC\times\CC^\op\to \BB
\]
together with dinatural transformations \(\beta' \colon P\dinat Q\) and
\(\beta\colon Q\dinat R\) which pair up the categories as pictured
below, then there is a composite \(\beta\circ\beta'\colon P\dinat R\).
 \[\begin{matrix}\raisebox{-.45\height}{\input{\figdir/H3_3_7.pstex_t}}\\\pstex{H3_3_8}\end{matrix}
    \mapsto \raisebox{-.45\height}{\input{\figdir/H3_3_9.pstex_t}}.\]
However, we can not form a dinatural transformation from the composite
\(\ev\circ \coev\) as we would get a loop in the Eilenberg-Kelly graph
as can be seen here:
 \[\raisebox{-.45\height}{\input{\figdir/H3_3_10.pstex_t}}.\]
See %\cite{Eilenberg} or
\cite{EilenbergKelly} for more details.

\subsection{Definition of a monoidal category with left duals}
We can now state the definition of a monoidal category with left duals in this language.  Suppose that $\CC$ is a monoidal category, ${}^\vee\colon \CC^{\op}\to \CC$ is a functor and $\ev\colon {}^\vee\otimes \Id_\CC\nattrans \Id_\CC$ and $\coev\colon \Id_\CC \nattrans  \Id_\CC\otimes{}^\vee$ are dinatural transformations drawn as
\[\ev\equiv \raisebox{-.45\height}{}\qquad \coev\equiv \raisebox{-.45\height}{}.\]
Then $({}^\vee,\ev,\coev)$ forms a \emph{left duality} of $\CC$ if the following snake relations hold:
\begin{align*}\raisebox{-.45\height}{}&=\raisebox{-.45\height}{\begin{picture}(0,0)%
\includegraphics{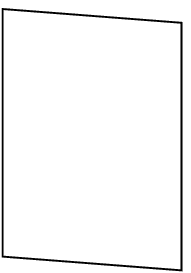}%
\end{picture}%
\setlength{\unitlength}{4144sp}%
\begingroup\makeatletter\ifx\SetFigFont\undefined%
\gdef\SetFigFont#1#2{%
  \fontsize{#1}{#2pt}%
  \selectfont}%
\fi\endgroup%
\begin{picture}(842,1219)(2684,-3068)
\end{picture}%
};\mytag{Duality1}\label{Duality1}\\
\raisebox{-.45\height}{\begin{picture}(0,0)%
\includegraphics{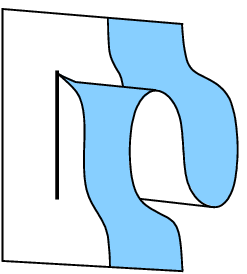}%
\end{picture}%
\setlength{\unitlength}{4144sp}%
\begingroup\makeatletter\ifx\SetFigFont\undefined%
\gdef\SetFigFont#1#2{%
  \fontsize{#1}{#2pt}%
  \selectfont}%
\fi\endgroup%
\begin{picture}(1099,1223)(2680,-3072)
\end{picture}%
}&=\raisebox{-.45\height}{\begin{picture}(0,0)%
\includegraphics{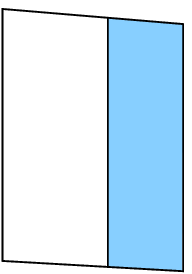}%
\end{picture}%
\setlength{\unitlength}{4144sp}%
\begingroup\makeatletter\ifx\SetFigFont\undefined%
\gdef\SetFigFont#1#2{%
  \fontsize{#1}{#2pt}%
  \selectfont}%
\fi\endgroup%
\begin{picture}(850,1223)(2680,-3072)
\end{picture}%
}.\mytag{Duality2}\label{Duality2}\end{align*}

\subsection{Composition with natural transformations}

\renewcommand{\tilde}{\widetilde}
Suppose that \(F\colon \tilde \AA\to \AA\),  \(G\colon \tilde \CC\to
\CC\),  \(H\colon \tilde \DD\to \DD\), and  \(K\colon \BB\to
\tilde\BB\) are functors  and that \(\beta \colon P\dinat Q\) is a
dinatural transformation of the above form then we get a dinatural
transformation
 \[\tilde \beta\colon K\circ P\circ (G\times G^\op\times F)\dinat
     K\circ Q\circ (F\times H^\op\times H),\]
given by
\(\tilde\beta_{\tilde a, \tilde c, \tilde d}:=K(\beta_{F(\tilde a),
  G(\tilde c), H(\tilde d)})\).
This is denoted graphically as
 \[\raisebox{-.45\height}{\input{\figdir/H3_4_1.pstex_t}}.\]

Note that in the case that \(\beta\) is an ordinary natural
transformation this recovers the ordinary horizontal composition of
natural transformations.

Suppose now that we have
natural transformations
 \[\raisebox{-.45\height}{\input{\figdir/H3_4_2.pstex_t}},\quad\raisebox{-.45\height}{\input{\figdir/H3_4_3.pstex_t}},\quad\raisebox{-.45\height}{\input{\figdir/H3_4_4.pstex_t}}\quad
    \text{and}\quad\raisebox{-.45\height}{\input{\figdir/H3_4_5.pstex_t}},\]
%%  functors \(F,\overline F\colon \tilde
%% \AA\to \AA\), \(G,\overline G\colon \tilde \CC\to \CC\), \(H,\overline
%% H\colon \tilde \DD\to \DD\), and \(K,\overline K\colon \BB\to
%% \tilde\BB\) and natural transformations
%% \(\phi\colon F\nattrans \overline F\), \(\gamma\colon G\nattrans
%% \overline G\), \(\eta\colon H\nattrans \overline H\) and \(\kappa\colon
%% K\nattrans \overline K\),
together with
a dinatural transformation
\(\beta \colon P\dinat Q\)
 of the above form
% \[\pichere\]
then the following dinatural
transformations are equal:
 \[\raisebox{-.45\height}{\input{\figdir/H3_4_7.pstex_t}}\quad=\quad\raisebox{-.45\height}{\input{\figdir/H3_4_8.pstex_t}}.\]
This just follows from the definitions.  In traditional notation it is
quite a mess to write down, thus this does show how nicely the
diagrammatics capture the essence of composition of dinatural
transformations.

\section{Hopf monads}

In this section we use the diagrammatic language to first give
Brugui\`eres and Virilizier's definition of a Hopf monad and a minor
simplification of their proof that such thing is equivalent to a lift
of duals to the module category. We then use the diagrammatics to be
more explicit than them in their example of a monad coming from a
strongly monoidal functor with a left adjoint.
\subsection{Hopf monads}

The difference between a bialgebra and a Hopf algebra is that the
latter has an antipode.  In the current context, the principal
consequence of this is that the vector space dual of a module over a
Hopf algebra carries a canonical action.  More precisely,
the duality on the base category of vector spaces lifts to a duality
on the category of modules.  It is this property that we wish to
examine for monads, and we can do this by asking the question ``What
structure is required of a bimonad on a monoidal category with duals
so that the category of modules has a lift of the duals?''  This was
answered by Brugui\`eres and Virelizier
\cite{BruguieresVirelizier:Hopf} and here we will describe their
solution in the diagrammatic language, something that they themselves
would have liked to have done.

Brugui\`eres and Virelizier \cite{BruguieresVirelizier:Hopf} gave the
definition of `left antipode'.  We will restate this definition using
the diagrammatic notation developed above.
\begin{defn} If \(T\) is a bimonad on a monoidal category with left
  duals then a \emph{left antipode} for \(T\) is a natural
  transformation \(\SSS\colon T\circ {}^\vee\circ T\nattrans {}^\vee\),
  denoted as follows,
\[\SSS\equiv \raisebox{-.45\height}{\begin{picture}(0,0)%
\includegraphics{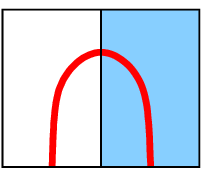}%
\end{picture}%
\setlength{\unitlength}{4144sp}%
\begingroup\makeatletter\ifx\SetFigFont\undefined%
\gdef\SetFigFont#1#2{%
  \fontsize{#1}{#2pt}%
  \selectfont}%
\fi\endgroup%
\begin{picture}(924,763)(1069,-362)
\put(1127,-253){\makebox(0,0)[lb]{\smash{{\SetFigFont{8}{9.6}{\color[rgb]{0,0,0}$T$}%
}}}}
\end{picture}%
},\]
which satisfies the following two relations:
\begin{align*}
\raisebox{-.45\height}{\begin{picture}(0,0)%
\includegraphics{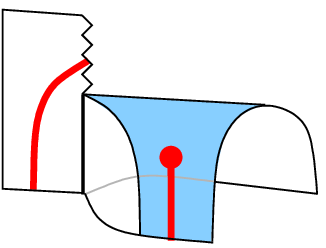}%
\end{picture}%
\setlength{\unitlength}{4144sp}%
\begingroup\makeatletter\ifx\SetFigFont\undefined%
\gdef\SetFigFont#1#2{%
  \fontsize{#1}{#2pt}%
  \selectfont}%
\fi\endgroup%
\begin{picture}(1464,1101)(2679,-3080)
\end{picture}%
}&=\raisebox{-.45\height}{\begin{picture}(0,0)%
\includegraphics{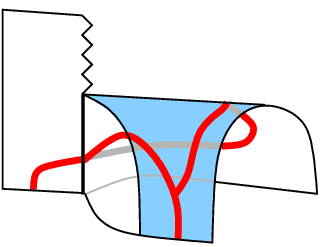}%
\end{picture}%
\setlength{\unitlength}{4144sp}%
\begingroup\makeatletter\ifx\SetFigFont\undefined%
\gdef\SetFigFont#1#2{%
  \fontsize{#1}{#2pt}%
  \selectfont}%
\fi\endgroup%
\begin{picture}(1464,1096)(2679,-3075)
\end{picture}%
};\mytag{HM1}\label{HM1}\\
\raisebox{-.45\height}{\begin{picture}(0,0)%
\includegraphics{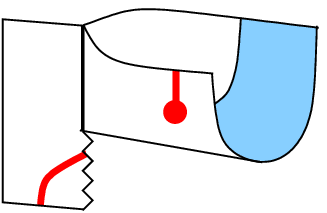}%
\end{picture}%
\setlength{\unitlength}{4144sp}%
\begingroup\makeatletter\ifx\SetFigFont\undefined%
\gdef\SetFigFont#1#2{%
  \fontsize{#1}{#2pt}%
  \selectfont}%
\fi\endgroup%
\begin{picture}(1468,943)(3391,-3353)
\end{picture}%
}&=\raisebox{-.45\height}{\begin{picture}(0,0)%
\includegraphics{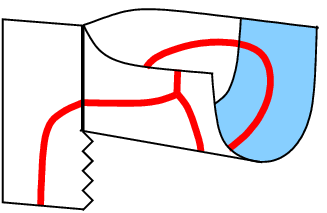}%
\end{picture}%
\setlength{\unitlength}{4144sp}%
\begingroup\makeatletter\ifx\SetFigFont\undefined%
\gdef\SetFigFont#1#2{%
  \fontsize{#1}{#2pt}%
  \selectfont}%
\fi\endgroup%
\begin{picture}(1468,943)(3391,-3353)
\end{picture}%
}.\mytag{HM2}\label{HM2}
\end{align*}
Here the dinatural transformation parts of the diagrams are the
evaluation and coevaluation of the duality on the category.

A bimonad equipped with a left antipode is called a (left) \emph{Hopf monad}.
\end{defn}

The question asked above is then fully answered by the following
theorem which tells us that, in a certain specific sense, Hopf monads
are analogous to Hopf algebras.
\begin{thm}[Brugui\`eres and Virelizier
  \cite{BruguieresVirelizier:Hopf}]\label{thm:antipodeduality}
  Suppose that \(T\) is a bimonad on a monoidal category \(\CC\) with
  a left duality and that \(T\CC\) is the monoidal category of
  \(T\)-modules. Then specifying a lift of the left duality on \(\CC\)
  to a left duality on \(T\CC\) is the same as specifying a left
  antipode for \(T\).
\end{thm}

The rest of this section will consist of a diagrammatic proof of the
above theorem.  We essentially translate Brugui\`eres and Virelizier's
proof into our diagrammatic language, with some minor simplification
making the proof more transparent.

We begin with a lemma similar to results in Section~\ref{section:bimonads} on bimonads.
\begin{lemma}
  If \(T\) is a monad on a category $\CC$ and ${}^\vee\colon\CC^\op\to
  \CC$ is a functor then lifts of this to a functor ${}^\wedge\colon
  T\CC^\op\to T\CC$ on the category of modules correspond to natural
  transformations \( T\circ {}^\vee\circ T\nattrans {}^\vee\), drawn
  as
\[\raisebox{-.45\height}{},\]
which satisfy
\[
  \raisebox{-.45\height}{\begin{picture}(0,0)%
\includegraphics{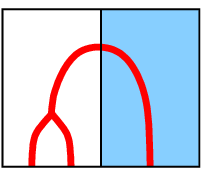}%
\end{picture}%
\setlength{\unitlength}{4144sp}%
\begingroup\makeatletter\ifx\SetFigFont\undefined%
\gdef\SetFigFont#1#2{%
  \fontsize{#1}{#2pt}%
  \selectfont}%
\fi\endgroup%
\begin{picture}(924,765)(1069,-364)
\put(1127,-253){\makebox(0,0)[lb]{\smash{{\SetFigFont{8}{9.6}{\color[rgb]{0,0,0}$T$}%
}}}}
\end{picture}%
}=\raisebox{-.45\height}{\begin{picture}(0,0)%
\includegraphics{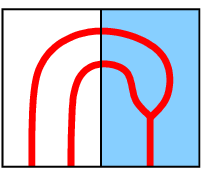}%
\end{picture}%
\setlength{\unitlength}{4144sp}%
\begingroup\makeatletter\ifx\SetFigFont\undefined%
\gdef\SetFigFont#1#2{%
  \fontsize{#1}{#2pt}%
  \selectfont}%
\fi\endgroup%
\begin{picture}(924,765)(1069,-364)
\end{picture}%
};
  \quad\text{and}\quad
  \raisebox{-.45\height}{\begin{picture}(0,0)%
\includegraphics{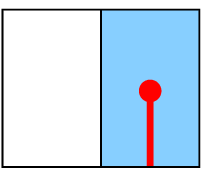}%
\end{picture}%
\setlength{\unitlength}{4144sp}%
\begingroup\makeatletter\ifx\SetFigFont\undefined%
\gdef\SetFigFont#1#2{%
  \fontsize{#1}{#2pt}%
  \selectfont}%
\fi\endgroup%
\begin{picture}(924,763)(1069,-362)
\end{picture}%
}=\raisebox{-.45\height}{\begin{picture}(0,0)%
\includegraphics{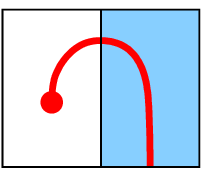}%
\end{picture}%
\setlength{\unitlength}{4144sp}%
\begingroup\makeatletter\ifx\SetFigFont\undefined%
\gdef\SetFigFont#1#2{%
  \fontsize{#1}{#2pt}%
  \selectfont}%
\fi\endgroup%
\begin{picture}(924,763)(1069,-362)
\end{picture}%
}.
  \mytag{HM0}\label{HM0}
  %\label{Eqn:SModuleConditions}
\]
\end{lemma}
\begin{proof}
The functor ${}^\wedge$ being a lift of the functor ${}^\vee$ means
that the following diagram commutes.
\[%\raisebox{25pt}
{\xymatrix%@C10pt
{
T\CC^\op
  \ar[r]^{{}^\wedge}
  \ar[d]_{U_T^\op}
  &T\CC
  \ar[d]^{U_T}\\
  \CC^\op
  \ar[r]^{{}^\vee}
  & \CC
}}\]
By
Lemma~\ref{Lemma:LiftingToModules}, specifying a functor
 ${}^\wedge\colon
  T\CC^\op\to T\CC$
is the same as specifying a functor  $U_T\circ{}^\wedge\colon
  T\CC^\op\to T\CC$ and a  natural transformation \(T\circ U_T\circ
{}^\wedge\nattrans U_T\circ {}^\wedge\) satisfying the two module
conditions.  But as the above diagram commutes, we know that  \(U_T\circ
{}^\wedge= {}^\vee\circ U_T^\op\), so we just need to specify a
natural transformation
% Suppose that we have a monoidal category \((\CC,{\otimes},\one)\)
% with left duals \(({}^\vee,\ev,\coev)\), and a bimonad \(T\) on
% \(\CC\).  We want to lift the left dual functor to a functor
% \({}^\wedge\colon T\CC^\op\to T\CC\), so here we use the same ideas as
% we did for bimonads, recall Lemma~\ref{Lemma:LiftingToModules} which
% implies that we just need to specify \(U_T\circ {}^\wedge\colon
% T\CC^\op\to \CC\) and a natural transformation \(T\circ U_T\circ
% {}^\wedge\nattrans U_T\circ {}^\wedge\) satisfying the two module
% conditions.  However, as this is going to be a lift of the dual on
% \(\CC\) we know that the object underlying the dual a module is the
% dual of the object underlying the module, or, in symbols, \(U_T\circ
% {}^\wedge= {}^\vee\circ U_T^\op\).  So to specify a lift is to specify
% a natural transformation
 \(T\circ {}^\vee\circ U_T^\op\nattrans
{}^\vee\circ U_T^\op\), drawn as
 \[\raisebox{-.45\height}{\input{\figdir/H4_1_1.pstex_t}},\]
and which satisfies the two module conditions.  Analogously to the bimonad
case we can use the identity \(T=U_T\circ F_T\) and the reversed
adjunction \(U_T^\op\dashv F_T^\op\) to obtain a bijection
\[
  \Nat(G\circ U_T^\op, H\circ U_T^\op)\cong
  \Nat(G\circ T^\op,H).
\]
So the above natural transformation corresponds to a natural
transformation \(\SSS\colon T\circ {}^\vee\circ T^\op\nattrans {}^\vee\):
\[
  \SSS \equiv\raisebox{-.45\height}{}:=\raisebox{-.45\height}{\input{\figdir/H4_1_3.pstex_t}}.
\]
The original natural transformation is recovered in the following way:
\[
  \raisebox{-.45\height}{\input{\figdir/H4_1_1.pstex_t}}=\raisebox{-.45\height}{\begin{picture}(0,0)%
\includegraphics{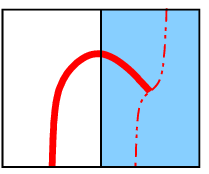}%
\end{picture}%
\setlength{\unitlength}{4144sp}%
\begingroup\makeatletter\ifx\SetFigFont\undefined%
\gdef\SetFigFont#1#2{%
  \fontsize{#1}{#2pt}%
  \selectfont}%
\fi\endgroup%
\begin{picture}(924,768)(1069,-362)
\put(1127,-253){\makebox(0,0)[lb]{\smash{{\SetFigFont{8}{9.6}{\color[rgb]{0,0,0}$T$}%
}}}}
\end{picture}%
}.
\]
Concretely, we can recover the lift \({}^\wedge\) from \(\SSS \) via
\[
    \pre{\wedge}(m,r)=\left(\prevee m,\ \SSS _{m}\circ T(\prevee
    r)\right).
\]
The two module conditions translate to \bref{HM0} as required.
\end{proof}
In order to show that such a lift of a left duality is itself a left duality we need to show that \(\ev\) and
\(\coev\) define \(T\)-module maps.  This is precisely where
\handtag{HM1} and \handtag{HM2} come into play.
% To this end,
% the following is a useful lemma.
% \begin{lemma}
%   Suppose that we have functors
%  \begin{gather*}
%   P\colon \DD\times\DD^\op\times \AA\to \CC
%    \quad\text{and} \quad
%   Q\colon \AA\times \BB^\op\times \BB\to \CC,
%  \intertext{and lifts}
%   \tilde P\colon \DD\times\DD^\op\times  \AA\to T\CC
%  \quad\text{and} \quad
%   \tilde Q\colon \AA\times \BB^\op\times \BB\to T\CC,
%  \end{gather*}
% meaning that \(U_T\circ \tilde P=P\) and  \(U_T\circ \tilde
%   Q=Q\).  Suppose further that we have a dinatural transformation
%   \(\beta\colon P\dinat Q\).  Then
%   \(\beta\) lifts to a dinatural transformation \(\tilde P\dinat \tilde
%   Q\) provided \(\beta\) commutes with the action in the following
%   sense:
%   \[
%     \pstex{H4_1_10}=\pstex{H4_1_11}.
%   \]
% \qed
% \end{lemma}
\begin{thm}\label{thm:HM1evalHM2coeval}
  Suppose that \(T\) is a bimonad on a monoidal category \(\CC\) with
  duals, and that \(\SSS \colon T\circ {}^\vee\circ T^\op\nattrans
  {}^\vee\) is a natural transformation satisfying the conditions
  \bref{HM0}, so it gives rise to a functor \({}^\wedge\colon
  T\CC^\op\to T\CC\).
\begin{enumerate}
\item The evaluation dinatural transformation on \(\CC\) lifts to a
  dinatural transformation on the module category \(T\CC\) if and only
  if \bref{HM1} is satisfied.
\item The coevaluation dinatural transformation on \(\CC\)
  lifts to a dinatural transformation on the module category \(T\CC\)
  if and only if \bref{HM2} is satisfied.
\end{enumerate}
\end{thm}
\begin{proof}
  Consider the evaluation case.  For \(\ev\) to lift to an evaluation
  on \(T\CC\) its components must be maps of \(T\)-modules, that is
  they must commute with the \(T\)-action, so for each \(T\)-module
  \((m,r)\in T\CC\) the following diagram must commute.
  \[
  \xymatrix{
    T(\prevee m\otimes m) \ar[d]\ar[r]^{T\ev_m}
    &T(\one)\ar[d]
    \\
    \prevee m \otimes m \ar[r]^{\ev_m}
    &\one
  }
  \]
  Here, of course, the action on \(\prevee m \otimes m\) is using
  \(\SSS\) and the bimonad structure.

%   By the previous lemma we require
%   \[\pstex{H4_1_16}=\pstex{H4_1_17}.\]
%   Now \(U_T\) commutes with everything in sight:
%   \(U_T\circ\one_{T\CC}=\one_\CC\), \(U_T\circ {\otimes}^T={\otimes}\circ
%   (U_T\times U_T)\), and \(U_T\circ {}^\wedge={}^\vee\circ U_T^\op\),
%   so a restatement of the above is
  In terms of the diagrammatic calculus this means that the following
  must hold:
  \[\raisebox{-.45\height}{\begin{picture}(0,0)%
\includegraphics{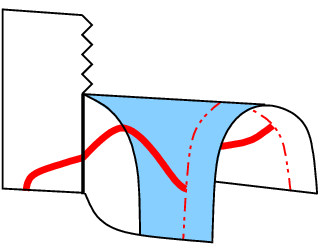}%
\end{picture}%
\setlength{\unitlength}{4144sp}%
\begingroup\makeatletter\ifx\SetFigFont\undefined%
\gdef\SetFigFont#1#2{%
  \fontsize{#1}{#2pt}%
  \selectfont}%
\fi\endgroup%
\begin{picture}(1464,1089)(2679,-3068)
\end{picture}%
}=\raisebox{-.45\height}{\begin{picture}(0,0)%
\includegraphics{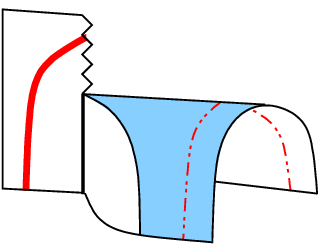}%
\end{picture}%
\setlength{\unitlength}{4144sp}%
\begingroup\makeatletter\ifx\SetFigFont\undefined%
\gdef\SetFigFont#1#2{%
  \fontsize{#1}{#2pt}%
  \selectfont}%
\fi\endgroup%
\begin{picture}(1464,1089)(2679,-3068)
\end{picture}%
}.\]
  Now we can use the same machinery as before to remove \(U_T\) from
  the statement: namely using we have a bijection
  \[\text{Dinat}\left(G\circ(U_T^\op \times U_T), H\right)
    \cong
    \text{Dinat}\left(G\circ(T^\op \times \Id), H\right).
  \]
  Via this correspondence, the above equality becomes
  \[\raisebox{-.45\height}{\begin{picture}(0,0)%
\includegraphics{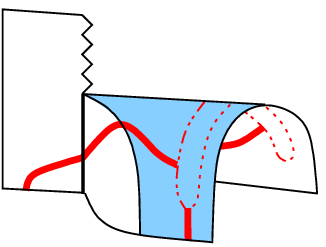}%
\end{picture}%
\setlength{\unitlength}{4144sp}%
\begingroup\makeatletter\ifx\SetFigFont\undefined%
\gdef\SetFigFont#1#2{%
  \fontsize{#1}{#2pt}%
  \selectfont}%
\fi\endgroup%
\begin{picture}(1464,1099)(2679,-3078)
\end{picture}%
}=\raisebox{-.45\height}{\begin{picture}(0,0)%
\includegraphics{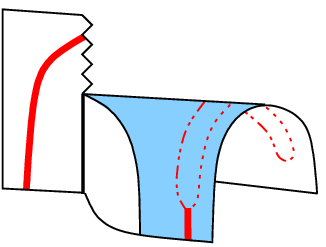}%
\end{picture}%
\setlength{\unitlength}{4144sp}%
\begingroup\makeatletter\ifx\SetFigFont\undefined%
\gdef\SetFigFont#1#2{%
  \fontsize{#1}{#2pt}%
  \selectfont}%
\fi\endgroup%
\begin{picture}(1464,1099)(2679,-3078)
\end{picture}%
},\]
  Moving things `over the top' this condition is seen to be equivalent to
  \[\raisebox{-.45\height}{\begin{picture}(0,0)%
\includegraphics{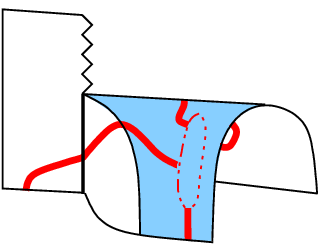}%
\end{picture}%
\setlength{\unitlength}{4144sp}%
\begingroup\makeatletter\ifx\SetFigFont\undefined%
\gdef\SetFigFont#1#2{%
  \fontsize{#1}{#2pt}%
  \selectfont}%
\fi\endgroup%
\begin{picture}(1464,1099)(2679,-3078)
\end{picture}%
}=\raisebox{-.45\height}{\begin{picture}(0,0)%
\includegraphics{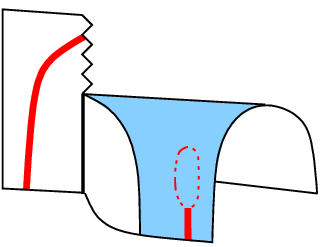}%
\end{picture}%
\setlength{\unitlength}{4144sp}%
\begingroup\makeatletter\ifx\SetFigFont\undefined%
\gdef\SetFigFont#1#2{%
  \fontsize{#1}{#2pt}%
  \selectfont}%
\fi\endgroup%
\begin{picture}(1464,1099)(2679,-3078)
\end{picture}%
},\]
  which is, by the properties of \(U_T\) and \(F_T\) from
  Section~\ref{Section:MonadsAndBimonads}, is just
  \[\raisebox{-.45\height}{}=\raisebox{-.45\height}{}\]
  and this \bref{HM1} as required.

  The coevaluation case is similar.
\end{proof}
  We have now seen that if \(T\) is a bimonad on a monoidal category with
  duals then specifying a lift of the duality on \(\CC\) to a duality
  on \(T\CC\) the category of \(T\)-modules is equivalent to
  specifying a natural transformation \( T\circ {}^\vee\circ
  T^\op\nattrans {}^\vee\) such that conditions \bref{HM0}, \bref{HM1} and \bref{HM2}
  are satisfied.  To prove Theorem~\ref{thm:antipodeduality} we just need to see that \bref{HM0} is actually a redundant condition.
\begin{thm}
  \label{Theorem:HM1HM2impliesHM0}
  Suppose that \(T\) is a bimonad on a monoidal category with duals,
  then any natural transformation  \( T\circ {}^\vee\circ
  T^\op\nattrans {}^\vee\) which satisfies \bref{HM1} and \bref{HM2} also
  satisfies \bref{HM0}.
\end{thm}
\begin{proof}
  We first prove that the following equation holds:
  \[\raisebox{-.45\height}{\begin{picture}(0,0)%
\includegraphics{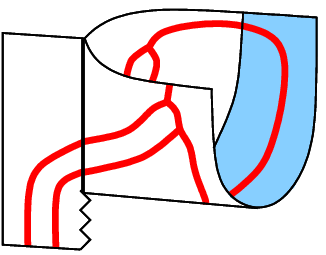}%
\end{picture}%
\setlength{\unitlength}{4144sp}%
\begingroup\makeatletter\ifx\SetFigFont\undefined%
\gdef\SetFigFont#1#2{%
  \fontsize{#1}{#2pt}%
  \selectfont}%
\fi\endgroup%
\begin{picture}(1461,1134)(2667,-4746)
\end{picture}%
}=\raisebox{-.45\height}{\begin{picture}(0,0)%
\includegraphics{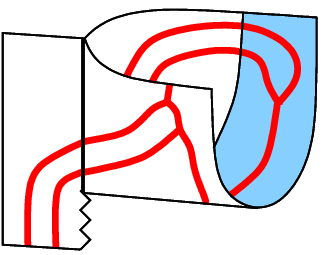}%
\end{picture}%
\setlength{\unitlength}{4144sp}%
\begingroup\makeatletter\ifx\SetFigFont\undefined%
\gdef\SetFigFont#1#2{%
  \fontsize{#1}{#2pt}%
  \selectfont}%
\fi\endgroup%
\begin{picture}(1461,1134)(2667,-4746)
\end{picture}%
}.\eqno{(\dagger)}\]
  This is true for the following reason:
  \begin{align*}
    \text{LHS}&:=\quad\raisebox{-.45\height}{\begin{picture}(0,0)%
\includegraphics{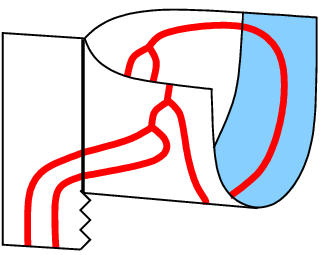}%
\end{picture}%
\setlength{\unitlength}{4144sp}%
\begingroup\makeatletter\ifx\SetFigFont\undefined%
\gdef\SetFigFont#1#2{%
  \fontsize{#1}{#2pt}%
  \selectfont}%
\fi\endgroup%
\begin{picture}(1461,1134)(2667,-4746)
\end{picture}%
}\quad
    \stackrel{\handtag{BM1}}=\quad\raisebox{-.45\height}{\begin{picture}(0,0)%
\includegraphics{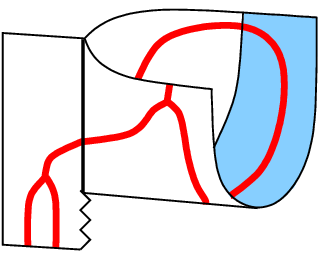}%
\end{picture}%
\setlength{\unitlength}{4144sp}%
\begingroup\makeatletter\ifx\SetFigFont\undefined%
\gdef\SetFigFont#1#2{%
  \fontsize{#1}{#2pt}%
  \selectfont}%
\fi\endgroup%
\begin{picture}(1461,1134)(2667,-4746)
\end{picture}%
}\\
    &\stackrel{\bref{HM2}}=\quad\raisebox{-.45\height}{\begin{picture}(0,0)%
\includegraphics{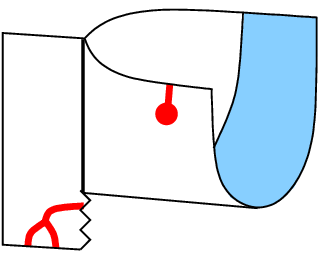}%
\end{picture}%
\setlength{\unitlength}{4144sp}%
\begingroup\makeatletter\ifx\SetFigFont\undefined%
\gdef\SetFigFont#1#2{%
  \fontsize{#1}{#2pt}%
  \selectfont}%
\fi\endgroup%
\begin{picture}(1461,1134)(2667,-4746)
\end{picture}%
}\quad
    \stackrel{\handtag{BM3}}=\quad\raisebox{-.45\height}{\begin{picture}(0,0)%
\includegraphics{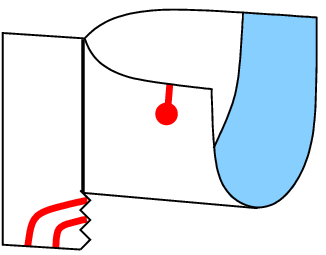}%
\end{picture}%
\setlength{\unitlength}{4144sp}%
\begingroup\makeatletter\ifx\SetFigFont\undefined%
\gdef\SetFigFont#1#2{%
  \fontsize{#1}{#2pt}%
  \selectfont}%
\fi\endgroup%
\begin{picture}(1461,1134)(2667,-4746)
\end{picture}%
}\\
    &\stackrel{\bref{HM2}}=\quad\raisebox{-.45\height}{\begin{picture}(0,0)%
\includegraphics{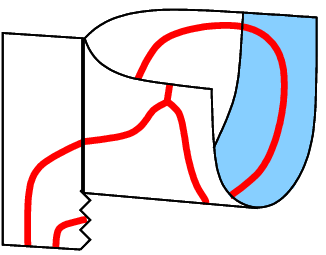}%
\end{picture}%
\setlength{\unitlength}{4144sp}%
\begingroup\makeatletter\ifx\SetFigFont\undefined%
\gdef\SetFigFont#1#2{%
  \fontsize{#1}{#2pt}%
  \selectfont}%
\fi\endgroup%
\begin{picture}(1461,1134)(2667,-4746)
\end{picture}%
}\quad
    \stackrel{\bref{Monad2}}=\quad\raisebox{-.45\height}{\begin{picture}(0,0)%
\includegraphics{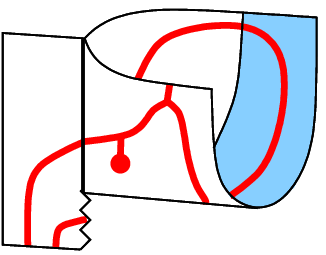}%
\end{picture}%
\setlength{\unitlength}{4144sp}%
\begingroup\makeatletter\ifx\SetFigFont\undefined%
\gdef\SetFigFont#1#2{%
  \fontsize{#1}{#2pt}%
  \selectfont}%
\fi\endgroup%
\begin{picture}(1461,1134)(2667,-4746)
\end{picture}%
}\\
    &\stackrel{\bref{HM2}}=\quad\raisebox{-.45\height}{\begin{picture}(0,0)%
\includegraphics{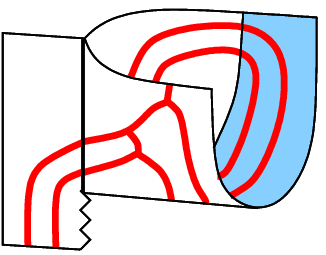}%
\end{picture}%
\setlength{\unitlength}{4144sp}%
\begingroup\makeatletter\ifx\SetFigFont\undefined%
\gdef\SetFigFont#1#2{%
  \fontsize{#1}{#2pt}%
  \selectfont}%
\fi\endgroup%
\begin{picture}(1461,1134)(2667,-4746)
\end{picture}%
}\quad
    \stackrel{\bref{Monad1}}=\quad\raisebox{-.45\height}{\begin{picture}(0,0)%
\includegraphics{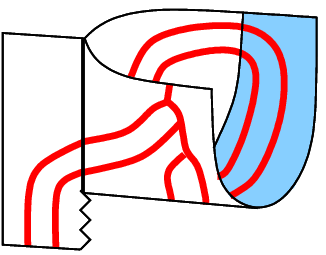}%
\end{picture}%
\setlength{\unitlength}{4144sp}%
\begingroup\makeatletter\ifx\SetFigFont\undefined%
\gdef\SetFigFont#1#2{%
  \fontsize{#1}{#2pt}%
  \selectfont}%
\fi\endgroup%
\begin{picture}(1461,1134)(2667,-4746)
\end{picture}%
}
    =\quad\text{RHS}.
  \end{align*}
 Also, by \handtag{HM1} we have
 \[\raisebox{-.45\height}{\begin{picture}(0,0)%
\includegraphics{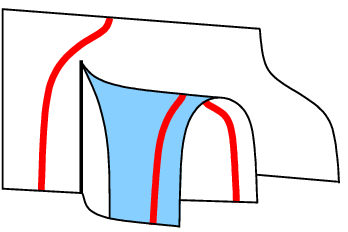}%
\end{picture}%
\setlength{\unitlength}{4144sp}%
\begingroup\makeatletter\ifx\SetFigFont\undefined%
\gdef\SetFigFont#1#2{%
  \fontsize{#1}{#2pt}%
  \selectfont}%
\fi\endgroup%
\begin{picture}(1562,1030)(2681,-3048)
\end{picture}%
}=\raisebox{-.45\height}{\begin{picture}(0,0)%
\includegraphics{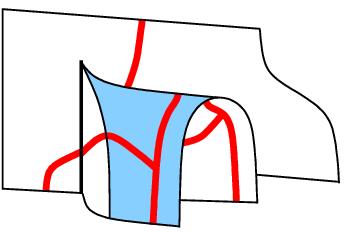}%
\end{picture}%
\setlength{\unitlength}{4144sp}%
\begingroup\makeatletter\ifx\SetFigFont\undefined%
\gdef\SetFigFont#1#2{%
  \fontsize{#1}{#2pt}%
  \selectfont}%
\fi\endgroup%
\begin{picture}(1562,1030)(2681,-3048)
\end{picture}%
}.\eqno{(\dagger\dagger)}\]
  Thus
  \begin{align*}
    \raisebox{-.45\height}{\begin{picture}(0,0)%
\includegraphics{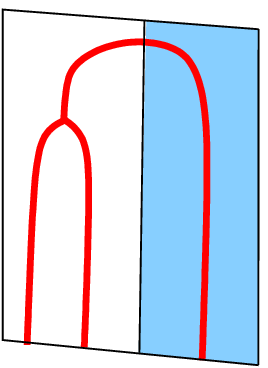}%
\end{picture}%
\setlength{\unitlength}{4144sp}%
\begingroup\makeatletter\ifx\SetFigFont\undefined%
\gdef\SetFigFont#1#2{%
  \fontsize{#1}{#2pt}%
  \selectfont}%
\fi\endgroup%
\begin{picture}(1194,1650)(2689,-3364)
\end{picture}%
}
    &\stackrel{\bref{Duality2}}=\raisebox{-.45\height}{\begin{picture}(0,0)%
\includegraphics{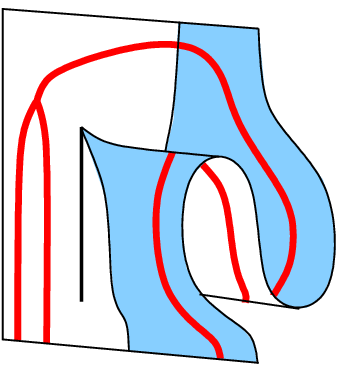}%
\end{picture}%
\setlength{\unitlength}{4144sp}%
\begingroup\makeatletter\ifx\SetFigFont\undefined%
\gdef\SetFigFont#1#2{%
  \fontsize{#1}{#2pt}%
  \selectfont}%
\fi\endgroup%
\begin{picture}(1543,1645)(2689,-3359)
\end{picture}%
}
     \stackrel{\handtag{$\dagger\dagger$}}\quad=\quad\raisebox{-.45\height}{\begin{picture}(0,0)%
\includegraphics{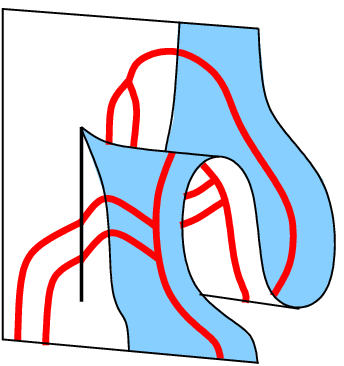}%
\end{picture}%
\setlength{\unitlength}{4144sp}%
\begingroup\makeatletter\ifx\SetFigFont\undefined%
\gdef\SetFigFont#1#2{%
  \fontsize{#1}{#2pt}%
  \selectfont}%
\fi\endgroup%
\begin{picture}(1543,1645)(2689,-3359)
\end{picture}%
}\\
     &\stackrel{\handtag{$\dagger$}}=\quad\raisebox{-.45\height}{\begin{picture}(0,0)%
\includegraphics{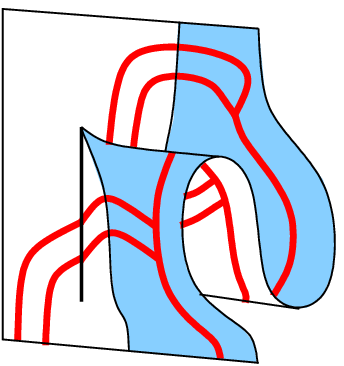}%
\end{picture}%
\setlength{\unitlength}{4144sp}%
\begingroup\makeatletter\ifx\SetFigFont\undefined%
\gdef\SetFigFont#1#2{%
  \fontsize{#1}{#2pt}%
  \selectfont}%
\fi\endgroup%
\begin{picture}(1543,1645)(2689,-3359)
\end{picture}%
}\quad
     \stackrel{\handtag{$\dagger\dagger$}}=\quad\raisebox{-.45\height}{\begin{picture}(0,0)%
\includegraphics{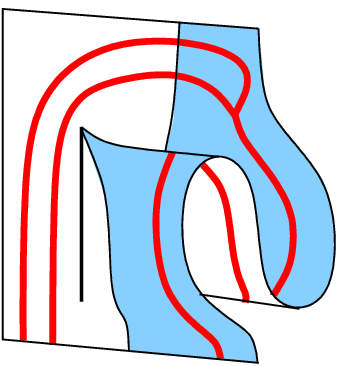}%
\end{picture}%
\setlength{\unitlength}{4144sp}%
\begingroup\makeatletter\ifx\SetFigFont\undefined%
\gdef\SetFigFont#1#2{%
  \fontsize{#1}{#2pt}%
  \selectfont}%
\fi\endgroup%
\begin{picture}(1543,1645)(2689,-3359)
\end{picture}%
}\\
     &\stackrel{\bref{Duality2}}=\quad\raisebox{-.45\height}{\begin{picture}(0,0)%
\includegraphics{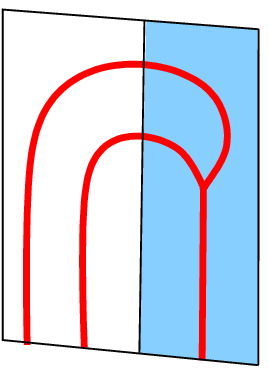}%
\end{picture}%
\setlength{\unitlength}{4144sp}%
\begingroup\makeatletter\ifx\SetFigFont\undefined%
\gdef\SetFigFont#1#2{%
  \fontsize{#1}{#2pt}%
  \selectfont}%
\fi\endgroup%
\begin{picture}(1194,1650)(2689,-3364)
\end{picture}%
}\quad.%\qedhere
  \end{align*}
\end{proof}
Now Theorem~\ref{thm:HM1evalHM2coeval} and Theorem~\ref{Theorem:HM1HM2impliesHM0} immediately imply Theorem~\ref{thm:antipodeduality} as required.

\subsection{Hopf monads from adjoint functors}
Suppose that there is an adjunction \(F\dashv U\) where
\(U\colon \DD\to \CC\) is a strong monoidal functor
between monoidal categories with (left) duals.  We know from Section~\ref{Subsection:BimonadsFromAdjoint}  that \(U\circ F\) is
a bimonad; we will see that it is also naturally a Hopf monad, i.e.,
that it naturally comes equipped with an antipode.  This is due to
Brugi\`eres and Virelizier but we make the structure more explicit
than in their paper.

\subsubsection{Overview}\label{section:HMAFoverview}
The key point for the definition of the antipode is that if \(U\) is strong monoidal then it commutes
with taking duals.  More precisely, we will see below that there is a natural isomorphism
\({}^\vee\circ U^\op \cong U\circ \prevee\) and we will draw the
mutually inverse
transformations as
\[\raisebox{-.45\height}{\begin{picture}(0,0)%
\includegraphics{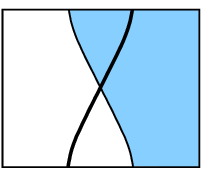}%
\end{picture}%
\setlength{\unitlength}{4144sp}%
\begingroup\makeatletter\ifx\SetFigFont\undefined%
\gdef\SetFigFont#1#2{%
  \fontsize{#1}{#2pt}%
  \selectfont}%
\fi\endgroup%
\begin{picture}(924,765)(1069,-353)
\put(1402,-190){\makebox(0,0)[rb]{\smash{{\SetFigFont{8}{9.6}{\color[rgb]{0,0,0}$U$}%
}}}}
\end{picture}%
}
\qquad\text{and}\qquad
\raisebox{-.45\height}{\begin{picture}(0,0)%
\includegraphics{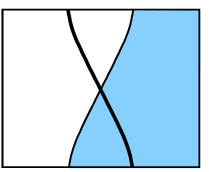}%
\end{picture}%
\setlength{\unitlength}{4144sp}%
\begingroup\makeatletter\ifx\SetFigFont\undefined%
\gdef\SetFigFont#1#2{%
  \fontsize{#1}{#2pt}%
  \selectfont}%
\fi\endgroup%
\begin{picture}(924,765)(1069,-354)
\put(1370,233){\makebox(0,0)[rb]{\smash{{\SetFigFont{8}{9.6}{\color[rgb]{0,0,0}$U$}%
}}}}
\end{picture}%
}.\]
We will also see below how to define these from the monoidal structure
of \(U\), but first we can use these together with the unit and counit
of the adjunction to
define  \(\SSS\colon U\circ F\circ{}^\vee\circ U^\op\circ
F^\op\nattrans{}^\vee\), the antipode for the bimonad \(U\circ F\), in the following way:
 \[\SSS :=\raisebox{-.45\height}{\input{\figdir/H4_2_3.pstex_t}}.\]
This will be shown to indeed be an antipode in Section~\ref{Subsubsection:SIsAntipode}, but the reader is invited to check diagrammatically that this satisfies \bref{HM0}.

\subsubsection{Strong monoidal functors commute with taking duals} We
will now show that \(U\) commutes with taking duals, i.e., that there exists a natural isomorphism \({}^\vee\circ U^\op \cong U\circ \prevee\).  On the level of objects
the idea is that for \(d\in \DD\) the objects \(\prevee U^\op(d)\) and
\(U(\prevee d)\) are both left duals of \(U(d)\), and so they are
canonically isomorphic via the standard yoga: namely, being somewhat
fastidious, we have
 \begin{align*}
   \prevee U^\op(d)
   &\to
      \prevee U^\op(d)\otimes \one
      \to
      \prevee U^\op(d)\otimes U(\one)
      \to
      \prevee U^\op(d)\otimes U(d\otimes \prevee d)
      \\
      &\to
      \prevee U^\op(d)\otimes \left(U(d)\otimes U(\prevee d)\right)
      \to
      \left( \prevee U^\op(d)\otimes U(d)\right)\otimes U(\prevee d)\\
      &\to
      \one \otimes U(\prevee d)
      \to
      U(\prevee d).
 \end{align*}
%%  \stackrel{\id\otimes U\otimes  \coev_d}{\longrightarrow}
%%  \prevee U^\op(d)\otimes U(d)\otimes U(\prevee d)
%%  \stackrel{\ev_{U(d)}\otimes \id}\longrightarrow
%%  U(\prevee d).
%%  \]
This gives rise to a natural isomorphism \({}^\vee\circ U^\op \nattrans U\circ
\prevee\) constructed from the dinatural transformations \(\ev\) and
\(\coev\) as follows:
 \[\raisebox{-.45\height}{}:=\raisebox{-.45\height}{\begin{picture}(0,0)%
\includegraphics{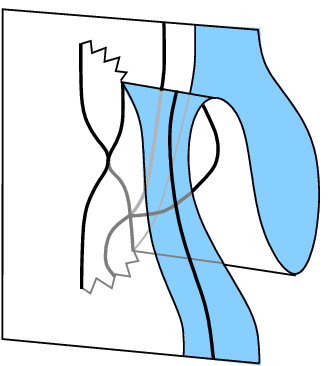}%
\end{picture}%
\setlength{\unitlength}{3947sp}%
\begingroup\makeatletter\ifx\SetFigFont\undefined%
\gdef\SetFigFont#1#2{%
  \fontsize{#1}{#2pt}%
  \selectfont}%
\fi\endgroup%
\begin{picture}(1544,1724)(3554,-3274)
\end{picture}%
}.\]
The inverse transformation is defined similarly.

Using these, we can now show that \(U\) also  commutes with evaluation
and coevaluation.
\begin{thm}\label{Thm:UcommutesEv}
  If \(U\) is a strong monoidal functor as above then \(U\) commutes
  with evaluation and coevaluation in the following sense:
  \[\raisebox{-.45\height}{\begin{picture}(0,0)%
\includegraphics{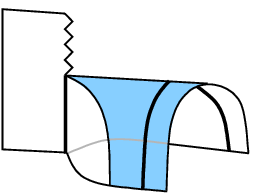}%
\end{picture}%
\setlength{\unitlength}{4144sp}%
\begingroup\makeatletter\ifx\SetFigFont\undefined%
\gdef\SetFigFont#1#2{%
  \fontsize{#1}{#2pt}%
  \selectfont}%
\fi\endgroup%
\begin{picture}(1150,857)(2993,-3068)
\put(3622,-2982){\makebox(0,0)[rb]{\smash{{\SetFigFont{8}{9.6}{\color[rgb]{0,0,0}$U$}%
}}}}
\end{picture}%
}=\raisebox{-.45\height}{\begin{picture}(0,0)%
\includegraphics{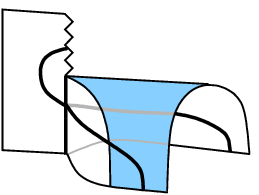}%
\end{picture}%
\setlength{\unitlength}{4144sp}%
\begingroup\makeatletter\ifx\SetFigFont\undefined%
\gdef\SetFigFont#1#2{%
  \fontsize{#1}{#2pt}%
  \selectfont}%
\fi\endgroup%
\begin{picture}(1155,860)(2988,-3068)
\put(3157,-2564){\makebox(0,0)[rb]{\smash{{\SetFigFont{8}{9.6}{\color[rgb]{0,0,0}$U$}%
}}}}
\end{picture}%
};\qquad\raisebox{-.45\height}{\begin{picture}(0,0)%
\includegraphics{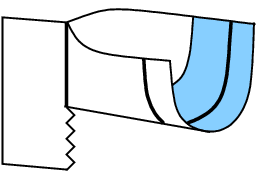}%
\end{picture}%
\setlength{\unitlength}{4144sp}%
\begingroup\makeatletter\ifx\SetFigFont\undefined%
\gdef\SetFigFont#1#2{%
  \fontsize{#1}{#2pt}%
  \selectfont}%
\fi\endgroup%
\begin{picture}(1180,760)(3785,-2834)
\put(4432,-2504){\makebox(0,0)[rb]{\smash{{\SetFigFont{8}{9.6}{\color[rgb]{0,0,0}$U$}%
}}}}
\end{picture}%
}=\raisebox{-.45\height}{\begin{picture}(0,0)%
\includegraphics{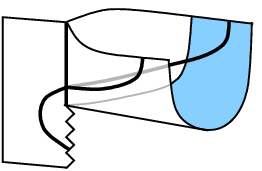}%
\end{picture}%
\setlength{\unitlength}{4144sp}%
\begingroup\makeatletter\ifx\SetFigFont\undefined%
\gdef\SetFigFont#1#2{%
  \fontsize{#1}{#2pt}%
  \selectfont}%
\fi\endgroup%
\begin{picture}(1169,752)(3329,-3463)
\put(3496,-3235){\makebox(0,0)[rb]{\smash{{\SetFigFont{8}{9.6}{\color[rgb]{0,0,0}$U$}%
}}}}
\end{picture}%
}.\]
  In more traditional notation this is expressing the commutativity of
  the following diagrams:
{\footnotesize\[\raisebox{25pt}{\xymatrix@C10pt{
  \prevee U^\op(d)\otimes U(d)
  \ar[rr]%^{\ev_{U(d)}}
  \ar[d]
  &
  &\one\\
  U(\prevee d)\otimes U(d)\ar[r]%^{U\ev_d}
  &  U(\prevee d\otimes d)\ar[r]
  &U(\one)\ar[u]%^\simeq
}};\quad\raisebox{25pt}{
\xymatrix@C10pt{
  \one
  \ar[rr]%^{\coev_{U(d)}}
  \ar[d]
  &
  & U^\op(d)\otimes\prevee U(d)\\
  U(\one)\ar[r]%^{U\coev_d}
  &  U( d\otimes \prevee d)\ar[r]
  &U( d)\otimes U(\prevee d)\ar[u]%^\simeq
}}
\]}

\end{thm}
%\begin{proof}
%  mm
%\end{proof}
%
%% It is not immediate that this gives a natural trnasformation as
%% \(\ev\) and \(\coev\) are not natural transformations, but it foes
%% follow from certain properties of \(\ev\) and \(\coev\) that we will
%% see later.  It is perhaps worth noting here that, that with this
%% identification, \(U\) ``commutes'' with evaluation and coevaluation.
%% We can just look at the evaluation case, where this amounts to
%% commutativity of the following diagram:
\begin{proof}
 We will consider the evaluation case as the coevaluation case is similar.  The left hand diagram is seen to commute as soon as its left hand arrow is unpacked as in the following diagram.\bigskip
 {\footnotesize
  \[\xymatrix@C10pt{
   \prevee U^\op(d)\otimes U(d)
   \ar@/^2em/[rr]%^{\ev_{U(d)}}
   \ar[d]%^\simeq
   \ar[rd]%^{\Id}
   \ar@/^4em/[rrdd]
   &&\one
   \\
   \prevee U^\op(d)\otimes U(\one)\otimes U(d)
   \ar[d]\ar[r]
   &
   \prevee U^\op(d)\otimes U(\one\otimes d)
   \ar[d]\ar[rd]
   \\
   \prevee U^\op(d)\otimes U(d\otimes\prevee d)\otimes U(d)
   \ar[d]\ar[r]
   &\prevee U^\op(d)\otimes U(d \otimes\prevee d \otimes d)
   \ar[d]\ar[r]
   &\prevee U^\op(d)\otimes U(d)\ar[uu]
   \\
   \prevee U^\op(d)\otimes U(d)\otimes U(\prevee d)\otimes U(d)
   \ar[d]
   \ar[r]
   & \prevee U^\op(d)\otimes U(d)\otimes U(\prevee d\otimes d)
   \ar[d]
   \\
   \one\otimes U(\prevee d)\otimes U(d)
   \ar[d]
   &\prevee U^\op(d)\otimes U(d)\otimes U(\one)
   \ar@/_2em/[uur]
   \\
   U(\prevee d)\otimes U(d)\ar[r]
   &
   U(\prevee d\otimes d)\ar[r]%^{U\ev_d}
   &U(\one)\ar@/_5em/[uuuuu]%^\simeq
 }
 \]}
\end{proof}

% Anyway, back to the antipode.  So if \(U\) has a left adjoint \(F\)
% then we can define the antipode \(S\colon U\circ F \circ {}^\vee\circ
% U\circ F\nattrans {}^\vee\) as below.
%  \[\pichere\]
% We will see later that the two relations on the antipode mentioned
% above are consequences of the other two involving the evaluation and
% coevaluation.  However as it is so easy, it is worth seeing that this
% antipode does commute with multiplication and the unit...

\subsubsection{The antipode and Hopf monad}
\label{Subsubsection:SIsAntipode}
We can now prove that the natural transformation \(\SSS\) defined above
does give a left antipode for the bimonad \(U\circ F\).  This is
essentially the proof of Theorem~3.14 in~\cite{BruguieresVirelizier:Hopf}.
\begin{thm}
  If \(F\dashv U\) is an adjunction where \(U\) is a strong monoidal
  functor between monoidal categories with left duals, then the
  natural transformation \(\SSS \colon T\circ{}^\vee\circ
  T^\op\nattrans {}^\vee\) (defined in
  Section~\ref{section:HMAFoverview}) is a left antipode for the
  bimonad \(U\circ F\).
\end{thm}
\begin{proof}  By Theorem~\ref{Theorem:HM1HM2impliesHM0}, it
  suffices to show that \bref{HM1} and \bref{HM2} are satisfied.  I will just
  give the proof of \bref{HM1}; the proof of \bref{HM2} is analogous.
\begin{align*}
  \raisebox{-.45\height}{\input{\figdir/H4_4_9.pstex_t}}&=\raisebox{-.45\height}{\input{\figdir/H4_4_10.pstex_t}}
  \stackrel{\text{Thm~\ref{Thm:UcommutesEv}}}=\raisebox{-.45\height}{\input{\figdir/H4_4_11.pstex_t}}\\
  &=\raisebox{-.45\height}{\input{\figdir/H4_4_12.pstex_t}}=\raisebox{-.45\height}{\input{\figdir/H4_4_13.pstex_t}}\\
  &\stackrel{\text{Thm~\ref{Thm:UcommutesEv}}}=\raisebox{-.45\height}{\input{\figdir/H4_4_14.pstex_t}}=\raisebox{-.45\height}{\input{\figdir/H4_4_15.pstex_t}}.
\end{align*}
\end{proof}
\noindent Thus we get that in this case \(U\circ F\) is indeed a Hopf
monad.


\begin{thebibliography}{99}
\bibitem{BruguieresVirelizier:Hopf} A.~Brugui\`eres and A.~Virelizier,
  Hopf monads, Advances in Mathematics, \textbf{215} (2007) 679-733.
%\bibitem{Eilenberg} S.~Eilenberg
\bibitem{EilenbergKelly}  S.~Eilenberg and M.~Kelly, A generalization
  of the functor calculus, Journal of Algebra, \textbf{3} (1966)
  366--375.
\bibitem{JoyalStreet} A.~Joyal and R.~Street, The geometry of tensor
  calculus I, Advances in Mathematics, \textbf{88} (1991) 55--113.
\bibitem{Lauda:FrobeniusAmbidextrous} A.~Lauda, Frobenius algebras and
  ambidextrous adjoints, Theory and Applications of Categories,
  \textbf{16}, no.~4 (2006) 84--122.
\bibitem{McCrudden:OpmonoidalMonads} P.~McCrudden, Opmonoidal monads,
  Theory and Applications of Categories, \textbf{10},  no.~19 (2002) 469--485.
\bibitem{McIntyreTrimble} M.~McIntyre and T.~Trimble, The geometry of
  Gray categories.
\bibitem{Moerdijk:MonadsTensorCategories} I.~Moerdijk, Monads on
  tensor categories, Journal of Pure and Applied Algebra,
  \textbf{168} (2002) 189--204.
\bibitem{Street:FunctorialCalculus} R.~Street,  Functorial Calculus in
  Monoidal Bicategories, Journal of Applied Categorical Structures,
  \textbf{11} (2003) 219--227.
\bibitem{Szlachanyi:MonoidalEM} K.~Szlach\'anyi, The monoidal
  Eilenberg-Moore construction and bialgebroids, Journal of Pure and
  Applied Algebra, \textbf{182} (2003) 287--315.
\end{thebibliography}
\end{document}